\documentclass[11pt]{article}       
\usepackage{geometry}               
\usepackage{graphicx}
\usepackage{caption}
\usepackage{subcaption}
\usepackage{amsmath} 
\usepackage{amssymb}  
\usepackage{amsthm}  
\usepackage{bm}
\usepackage{color}
\usepackage{array}
\usepackage{cite}
\usepackage{hyperref}
\usepackage{comment}

\newtheorem{proposition}{Proposition}

\newtheorem{corollary}{Corollary}

\newtheorem{theorem}{Theorem}

\numberwithin{equation}{section}

\DeclareMathOperator*{\essinf}{ess\,inf}

\title{Periodic homogenization of geometric equations without perturbed correctors}

\newcommand{\footremember}[2]{%
    \footnote{#2}
    \newcounter{#1}
    \setcounter{#1}{\value{footnote}}%
}

\author{%
  Jiwoong Jang \footremember{trailer}{Department of Mathematics, University of Maryland-College Park (jjang124@umd.edu)}
  }

\date{}
\providecommand{\keywords}[1]{\textbf{Key words.} #1}
\providecommand{\MSC}[1]{\textbf{MSC codes.} #1}

\begin{document}
\maketitle

\pagestyle{myheadings}
\thispagestyle{plain}

\begin{abstract}
Proving homogenization has been a subtle issue for geometric equations due to the discontinuity when the gradient vanishes. A sufficient condition for periodic homogenization using perturbed correctors is suggested in the literature \cite{CM14} to overcome this difficulty. However, some noncoercive equations do not satisfy this condition. In this note, we prove homogenization of geometric equations without using perturbed correctors, and therefore we conclude homogenization for the noncoercive equations. Also, we provide a rate of periodic homogenization of coercive geometric equations by utilizing the fact that they remain coercive under perturbation. We also present an example that homogenizes with a rate slower than $O(\varepsilon)$.
\end{abstract}


\vspace{0.2cm}

\MSC{35B10, 35B27, 35B40, 35D40}

\section{Introduction}\label{sec:introduction}
\subsection{Settings and motivations}\label{subsec:settings}
In this paper, we are interested in homogenization of geometric equations in the periodic setting, i.e., the convergence of solutions $u^{\varepsilon}(x,t)$ to
\begin{align}\label{eq:uepsilon}
\begin{cases}
u^{\varepsilon}_t+F\left(\varepsilon D^2 u^{\varepsilon},Du^{\varepsilon},\frac{x}{\varepsilon}\right)=0, \quad &\text{in}\quad\mathbb{R}^n\times(0,\infty),\\
u^{\varepsilon}(\cdot,0)=u_0, \quad &\text{on}\quad\mathbb{R}^n,
\end{cases}
\end{align}
to the solution $u(x,t)$ to
\begin{align}\label{eq:ueffective}
\begin{cases}
u_t+\overline{F}\left(\frac{Du}{|Du|}\right)|Du|=0, \quad &\text{in}\quad\mathbb{R}^n\times(0,\infty),\\
u(\cdot,0)=u_0, \quad &\text{on}\quad\mathbb{R}^n.
\end{cases}
\end{align}
as $\varepsilon\to0^+$, where $F$ is an operator periodic in the spatial variable, and $\overline{F}$ is an effective operator. The equations \eqref{eq:uepsilon}, \eqref{eq:ueffective} describe the large-scale behavior of the level-sets $\{u^{\varepsilon}(\cdot,t)=0\}$, understood as the fronts. The environment and the curvature determine the rule of evolution of the fronts as the normal velocity in typical models. As the curvature effect is now involved, we call the operator $F$ \emph{geometric}, of which we will provide the precise definition in Subsection \ref{subsec:mainresults}.

Our main motivations of this study are the following two geometric equations; one is the curvature $G$-equation \cite{GLXY22} with the normal velocity $V=(1-\varepsilon d\kappa)_++W\left(\frac{x}{\varepsilon}\right)\cdot\Vec{\mathrm{n}}$ with the microscale parameter $\varepsilon>0$. Here, the normal vector is outward to the set $\{u^{\varepsilon}(\cdot,t)>0\}$, and $\kappa$ represents the mean curvature, which is nonpositive when $\{u^{\varepsilon}(\cdot,t)>0\}$ is convex. The number $d>0$ describes the flame thickness \cite{M51,P00}, and $W$ is a vector field modeling wind. The other is the mean curvature equation with the normal velocity $V=\varepsilon\kappa+c\left(\frac{x}{\varepsilon}\right)$ with a spatial forcing term $c$.

It was first pointed out in \cite{CM14} that the classical perturbed test function method \cite{E89,E92} does not directly imply the full homogenization for geometric equations because of the discontinuity of geometric operators $F$ at $p=0$. Instead, the conditional homogenization \cite[Theorem 1.5]{CM14} is derived under a stronger condition, namely, that \emph{perturbed correctors} exist. This condition is satisfied by coercive forced equations, as the perturbation respects the coercivity condition on $c$. However, this is not the case for the curvature $G$-equation due to the noncoercive nature, meaning that homogenization of the curvature $G$-equation has been unclear so far. Motivatived by this circumstance, we relax the condition on perturbed correctors in this paper.

\subsection{Literature overview}\label{subsec:literature}

Homogenization of geometric equations in periodic media has received a lot of attentions. In \cite{LS05}, it is shown that mean curvature motions with a forcing term $c$ admit Lipschitz continuous correctors under the coercivity condition that
\begin{align}\label{assumption:coercivity}
\essinf_{\mathbb{R}^n}\left(c^2-(n-1)|Dc|\right)\geq\delta>0
\end{align}
holds for some $\delta>0.$ This condition can be seen as a small oscillation condition on $|Dc|$, depending on the magnitude of a positive force. For a positive forcing term without assuming the condition \eqref{assumption:coercivity}, \cite{CM14} shows by an example that homogenization does not happen in 3-d or higher. In contrast, also shown in \cite{CM14}, homogenization holds true in 2-d as long as the force is positive. In \cite{GK20}, a further asymptotic analysis on the head speed and the tail speed is given for a positive and Lipschitz continuous force. It is concluded in \cite{GK20} that the head and tail speeds are continuous in the normal directions, and homogenization happens for any uniformly continuous initial data if and only if the two speeds agree.

The case of sign-changing forces also has been investigated with smallness conditions. Namely, \cite{DKY08} found Lipschitz continuous correctors under the condition that $c\in C^2(\mathbb{T}^n)$ and $\|c\|_{C^2(\mathbb{T}^n)}$ is small enough. A variational approach is taken in \cite{CN13}, showing that if $c(x)=g(x_1,\cdots,x_{n-1})$ (a laminated environment), and if
\begin{align*}
\int_{(0,1)^{n-1}}g>0,\ \min g\leq0\ \text{and }\max g-\min g< C_n 2^{1/n},
\end{align*}
with the isoperimetric constant $C_n>0$ (appearing in \cite{BBBL03}), then there exists a Lipschitz continuous corrector for $p=e_n$. In general, moreover, \cite{CN13} proved the existence of \emph{generalized} traveling waves whose support is not necessarily a full cell. Also, interesting questions about homogenization with a sign-changing force, together with first-order Hamilton-Jacobi equations without curvature effect, are discussed in \cite{CLS09}. The paper \cite{CLS09} proved that when $n=2,\ c(x)=g(x_1)$, there exists a Lipschitz continuous corrector under the condition
\begin{align*}
\int_0^1g>0\ \text{and }\int_0^1g-\min g<2.
\end{align*}
Allowing a large oscillation, a counterexample to the homogenization with $\int_0^1g=0$ is also given in \cite{CLS09}.

Another interesting class of geometric equations in the context of homogenization is the curvature $G$-equations. The equations have been introduced in \cite{M51} and been mathematically studied in \cite{GLXY22,MMTXY23} recently. The curvature effect in $V=(1-\varepsilon d\kappa)_++W\left(\frac{x}{\varepsilon}\right)\cdot\Vec{\mathrm{n}}$ is considered to describe the physical phenomenon of the flame propagation that a concave part of the flame front propagates faster proportionally to the flame thickness, called the Markstein number \cite{M51}. The ()$_+$-correction is considered in \cite{GLXY22,MMTXY23} in order to ensure the physical validity, and this correction was first introduced in \cite{R94}.

\vspace{0.5cm}


One of the main points of \cite{CM14} is about the issue on deriving homogenization for general geometric equations, including the forced mean curvature motions and the curvature $G$-equations described above. The paper \cite{CM14} carefully applied the perturbed test function method \cite{E89,E92}, and provided a sufficient condition to guarantee the full homogenization, namely that the cell problems are solved for perturbed equations. Mean curvature motions with a coercive forcing term $c$, satisfying either \eqref{assumption:coercivity} in all dimensions or $\inf c>0$ in 2-d, meet this sufficient condition, and consequently, homogenization is concluded. However, it was unclear whether the curvature $G$-equations studied in \cite{GLXY22} satisfy the condition or not.




Facing this difficulty, it is natural to ask whether or not having correctors, not including perturbed ones, ensures homogenization. This is a nontrivial issue for geometric equations as pointed out in \cite{CM14}, although commonly believed to be true from the perturbed test function method \cite{E89,E92}. In this paper, we confirm that this belief is true even for geometric equations. The idea is that we use \emph{perturbed approximate correctors} instead of perturbed correctors, which is in line with the use of approximate correctors in \cite{C-DI01}.

Using the fact that mean curvature motions with a coercive forcing term satisfying \eqref{assumption:coercivity} enjoy Lipschitz estimates including their perturbed equations, we give a rate of homogenization of the motions in the periodic setting in this note. See \cite{AC18} in this direction for the random setting. A rate $O(\varepsilon^{1/2})$ for the case of periodic, laminated media can be obtained with a simpler estimate \cite{J23}, and we refer to \cite{QSTY23} for the case of viscous Hamilton-Jacobi equations. Obtaining an optimal rate is a different issue, and see \cite{TY23} for the very recent development of the case of first-order convex Hamilton-Jacobi equations.

\subsection{A brief description of the proof of Theorem \ref{thm:homogenizationofgeometricequations}}\label{subsec:briefdescription}
In this subsection, we give an outline of the proof of Theorem \ref{thm:homogenizationofgeometricequations}, which will be introduced in the next subsection as our main result. To achieve this, we first explain the use of the perturbed test function method in \cite{CM14}. 
The paper \cite{CM14} points out a subtle issue during the classical use of the method as suggested in \cite{E89,E92}. To explain the issue briefly, we implement the classical framework first. We let $\varphi$ be a test function that $\overline{u}-\varphi$ attains a (strict) maximum at $P_0$, where $\overline{u}:=\limsup_{\varepsilon\to0}^*u^{\varepsilon}$. For the standard framework with the doubling variable, we consider
\[\Phi(x,y,t):=u^{\varepsilon}(x,t)-\varphi(y,t)-\varepsilon v\left(\frac{y}{\varepsilon}\right)-\frac{|x-y|^2}{2\sigma}\quad\text{ for }\varepsilon,\sigma>0.\]
Here, a function $v$ is a corrector associated with $p_0=D\varphi(P_0)$. The last term is a penalization term, and we take $\sigma\to0$ first in this framework so that $\lim_{\varepsilon\to0}\lim_{\sigma\to0}(x^{\varepsilon}_{\sigma},y^{\varepsilon}_{\sigma},t^{\varepsilon}_{\sigma})=(x_0,x_0,t_0)$. At a maximizer $(x^{\varepsilon}_{\sigma},y^{\varepsilon}_{\sigma},t^{\varepsilon}_{\sigma})$ of $\Phi$, we find appropriate jets of $u^{\varepsilon}$ and of $v$ by the Crandall-Ishii Lemma, for the latter of which jets of $\varphi$ are translated. We subsequently obtain two viscosity inequalities, one from the subsolution test of $u^{\varepsilon}$ and the other from the supersolution test of $v$. We substract them to get a form
\[
\varphi_t(x_{\sigma}^{\varepsilon},t_{\sigma}^{\varepsilon})+F_*\left(\varepsilon Y,p_1,\frac{x_{\sigma}^{\varepsilon}}{\varepsilon}\right)-F^*\left(\varepsilon Y,p_2,\frac{y_{\sigma}^{\varepsilon}}{\varepsilon}\right)+\overline{F}(p_0)\leq0
\]
with derivatives satisfying
\begin{align*}
\begin{cases}
-\frac{M\varepsilon}{\sigma}I_n\leq \varepsilon Y\leq\frac{M\varepsilon}{\sigma}I_n,\\
\lim_{\varepsilon\to0}\lim_{\sigma\to0}\left(p_2-p_1\right)=0.
\end{cases}
\end{align*}
We can complete the proof if the two terms in the middle are canceled out during the limit $\sigma\to0,\ \varepsilon\to0$ in turn. However, for geometric operators such as $F(X,p)=-\mathrm{tr}\left\{\left(I_n-\widehat{p}\otimes\widehat{p}\right)X\right\}$, the difference of the two terms is $-\mathrm{tr}\left\{\left(\widehat{p_2}\otimes\widehat{p_2}-\widehat{p_1}\otimes\widehat{p_1}\right)\varepsilon Y\right\}$, here $\widehat{p}=\frac{p}{|p|}$ for $p\neq0$, for which we face two issues during the limit. The first is that we lose a bound of $\varepsilon Y$ as $\sigma\to0$, and the second is that $\widehat{p_2}$ does not necessarily approximate $\widehat{p_1}$ as $\sigma\to0,\ \varepsilon\to0$ in order.

The paper \cite{CM14} imposes an additional condition to resolve the issues, which namely assumes that we can find a \emph{perturbed corrector} $v^{2\eta}$ with uniform oscillation $\kappa_0$, i.e., a solution to
\begin{align*}
\begin{cases}
    F^{2\eta}(D^2v^{2\eta},p_0+Dv^{2\eta},y)=\overline{F}^{2\eta}(p_0)\quad\text{ on }\mathbb{T}^n,\\
    \sup v^{2\eta}-\inf v^{2\eta}\leq \kappa_0\,\quad\qquad\qquad\qquad\text{ for all }\eta\in[0,\eta_0].
\end{cases}
\end{align*}
for some $\kappa_0,\eta_0>0$. Here, $F^{2\eta}$ is defined as in \eqref{line:perturbedoperators}. We take $\sigma=\frac{1}{2}\kappa_0^{-1}\varepsilon\eta^2$ so that $-\frac{M\kappa_0}{\eta^2}I_n\leq \varepsilon Y\leq\frac{M\kappa_0}{\eta^2}I_n$. Also, the approximation $\lim_{\varepsilon\to0}\left(\widehat{p_2}-\widehat{p_1}\right)=0$ (now without $\sigma>0$) is concluded as the gradients are now separated from the origin uniformly in $\varepsilon>0$. This is achieved with a geometric idea using the inf-convolution by balls, for which we refer to \cite[Lemma 13.1]{CM14}.

Perturbed correctors are found by using coercivity conditions, such as \eqref{assumption:coercivity} or $\inf c\geq\delta>0$ in 2-d for forced mean curvature motions. These conditions are stable under perturbations (see \cite{LS05}, Proposition \ref{prop:correctors}, \cite[Proposition 9.1]{CM14}), which results in uniform oscillation $\kappa_0>0$ of perturbed correctors. Here, the amplitude $\kappa_0>0$ depends on a coercivity parameter $\delta>0$. The estimates of $\kappa_0$ shown so far in the literature, to the best of the author's knowledge, blow up as $\delta\to0$. In other words, as we lose coercivity, it is hard to expect the solvability of perturbed correctors and their oscillation stable under perturbations. In particular, whether or not we have perturbed correctors for the curvature G-equations is unclear.

A natural way to look at corrector problems is via their approximate problems with a monotone term. We then have the uniqueness of approximate correctors thanks to the monotone term, and moreover, approximate correctors $v^{\lambda}=v^{\lambda}(\cdot,p)$ enjoy a regularity as associated gradients $p$ vary. This regularity is utilized in the quantitative result in \cite{C-DI01}, and this is the reason why approximate correctors are used instead of correctors. That is, it is hard to expect the regularity of correctors $v=v(\cdot,p)$ with respect to $p$ in general. It indeed turns out later that there exists a Hamiltonian, constructed in \cite{MTY19}, for which there is no such regular selection of correctors. 

Our approach of this paper for the proof of Theorem \ref{thm:homogenizationofgeometricequations} is based on this philosophy. As we no more expect perturbed correctors to exist, we instead use \emph{perturbed approximate correctors}, which are always accessible thanks to the monotone term. This is a natural consideration in view of the above context, and moreover, our approach seems to be a correct viewpoint for geometric equations. If we think of the standard framework in which we separate the corrector problem (local problem) from the equation with $u^{\varepsilon}$ (global problem) when proving homogenization in general, our approach geometrically makes sense in that in our situation with geometric operators, all the parameters from the local problem, now the perturbation parameter $\eta$ being included, are independent of $\varepsilon$, which then turn out to provide a uniform lower bound for the separation of vanishing gradient.

\subsection{Main results}\label{subsec:mainresults}

Let $n\geq2$ throughout this paper. Let $S^n$ denote the set of $n\times n$ symmetric matrices. We consider operators $F:S^n\times (\mathbb{R}^n\setminus\{0\})\times \mathbb{R}^n\to\mathbb{R}$ satisfying the following properties.

\begin{itemize}
    \item[(I)] $F$ is continuous in $S^n\times (\mathbb{R}^n\setminus\{0\})\times \mathbb{R}^n$.

    \item[(II)] $F$ is \emph{degenerate elliptic}; $F(Y,p,y)\leq F(X,p,y)$ for all $X,Y\in S^n$ with $X\leq Y$ and all $(p,y)\in(\mathbb{R}^n\setminus\{0\})\times \mathbb{R}^n$.

    \item[(III)] $F$ is \emph{geometric}; $F(\lambda X+\mu p\otimes p, \lambda p,y)=\lambda F(X,p,y)$ for all $(X,p,y)\in S^n\times (\mathbb{R}^n\setminus\{0\})\times \mathbb{R}^n$ and all $\lambda>0,\ \mu\in\mathbb{R}$.

    \item[(IV)] $F$ is $\mathbb{Z}^n$-periodic; $F(X,p,y+k)=F(X,p,y)$ for all $(X,p,y)\in S^n\times (\mathbb{R}^n\setminus\{0\})\times \mathbb{R}^n$ and $k\in\mathbb{Z}^n$.

    \item[(V)] $F$ is \emph{regular};
    \begin{itemize}
        \item[(i)] For every $R>0$, there exists $M>0$ such that $|F(X,p,y)|\leq M$ for all $(X,p,y)\in S^n\times (\mathbb{R}^n\setminus\{0\})\times \mathbb{R}^n$ with $\|X\|\leq R,\ 0<|p|\leq R.$
        \item[(ii)] There exist $K>0$ and $\omega : [0,\infty)\to[0,\infty)$ such that $\omega(0^+)=0$ and
        $$F^*(X,\alpha(x-y),x)-F_*(Y,\alpha(x-y),y)\leq\omega\left(|x-y|(1+\alpha|x-y|)\right)$$
        for all $\alpha\geq0,\ X,Y\in S^n,\ x,y\in\mathbb{R}^n$ satisfying
        $$-K\alpha I_{2n}\leq\begin{pmatrix}X &0 \\0 &-Y \end{pmatrix}\leq K\alpha\begin{pmatrix}I_n &-I_n \\-I_n &I_n \end{pmatrix}$$ with $\alpha=0$ when $x=y$.
    \end{itemize}
\end{itemize}

The notations appearing in the above conditions are explained before Section \ref{sec:geqn}.

The condition (V) is a technical condition which guarantees the comparison principle for \eqref{eq:uepsilon}. For the statement of the principle and its proof, we refer to \cite[Theorem 3.3, Subsection 13.2]{CM14}. Regarding the comparison principle for \eqref{eq:ueffective} (also for \eqref{eq:heffective} as well), the principle holds true as long as the initial datum is uniformly continuous on $\mathbb{R}^n$. We remark that we do not need growth conditions for \eqref{eq:uepsilon} due to the condition (I) of $F$ (or the 1-positive homogeneity of $\overline{F}$ for \eqref{eq:ueffective}). We skip the proof of the comparison principle in this paper as it is standard.

Now we state the main theorem.

\begin{theorem}\label{thm:homogenizationofgeometricequations}
Suppose that a given operator $F:S^n\times (\mathbb{R}^n\setminus\{0\})\times \mathbb{R}^n\to\mathbb{R}$ satisfies the conditions (I)--(V). Suppose that for each $p\in\mathbb{R}^n$, there exists a unique real number $\overline{F}(p)$ such that
\begin{align}
F(D^2v,p+Dv,y)=\overline{F}(p)\qquad\textrm{on }\mathbb{R}^n\label{eq:cellproblem}
\end{align}
admits a $\mathbb{Z}^n$-periodic viscosity solution $v:\mathbb{R}^n\to\mathbb{R}$. Then, homogenization takes place, that is, $u^{\varepsilon}$ converges to $u$ locally uniformly on $\mathbb{R}^n\times[0,\infty)$ as $\varepsilon\to0$, where $u^{\varepsilon}$ and $u$ are the unique viscosity solution to \eqref{eq:uepsilon} and to \eqref{eq:ueffective}, respectively, and $u_0$ represents a given uniformly continuous function on $\mathbb{R}^n$.
\end{theorem}

It is noteworthy that Theorem \ref{thm:homogenizationofgeometricequations} applies also to the forced mean curvature flow under assumptions which have not been considered in Subsection \ref{subsec:literature} of literature, as the equation satisfies the conditions (I)--(V). It is also noteworthy that for the cases of nonhomogenization discussed in Subsection \ref{subsec:literature} of literature, there are some directions $p\in\mathbb{R}^n$ for which there does not exist a $\mathbb{Z}^n$-periodic viscosity solution $v:\mathbb{R}^n\to\mathbb{R}$ to \eqref{eq:cellproblem}.

As a corollary, we conclude homogenization of the curvature $G$-equation for the two dimensional cellular flow. Let us first state the main result of \cite[Theorem 1.1]{GLXY22} on the effective burning velocity and the uniform flatness.

\begin{theorem}\label{thm:effectiveburningvelocity}\cite[Theorem 1.1]{GLXY22}
For $\varepsilon>0,\ p\in\mathbb{R}^2$, let $u^{\varepsilon}$ denote the unique viscosity solution to
\begin{align}\label{eq:epsilong}
u^{\varepsilon}_t+\left(1-\varepsilon d\mathrm{div}\left(\frac{Du^{\varepsilon}}{|Du^{\varepsilon}|}\right)\right)_+|Du^{\varepsilon}|+V\left(\frac{x}{\varepsilon}\right)\cdot Du^{\varepsilon}=0, \quad &\text{in}\quad\mathbb{R}^2\times(0,\infty)
\end{align}
with $u^{\varepsilon}(x,0)=p\cdot x,\ x\in\mathbb{R}^2$, where $d>0$, and $V(x)=A(-\cos x_2\sin x_1,\cos x_1\sin x_2)$ is the two dimensional cellular flow with the flow intensity $A>0$ for $x=(x_1,x_2)\in\mathbb{R}^2$. Then, there exists a unique real number $\overline{H}(p)$ such that
\begin{align*}
|u^{\varepsilon}(x,t)-p\cdot x+\overline{H}(p)t|\leq C\varepsilon\qquad\textrm{on }\mathbb{R}^2\times[0,\infty),
\end{align*}
for some constant $C>0$ depending only on $d,A,|p|$. Moreover, the map $p\mapsto\overline{H}(p)$ is a continuous, 1-positive homogeneous function from $\mathbb{R}^2\setminus\{0\}$ to $(0,\infty)$.
\end{theorem}

As mentioned in \cite[Section 4]{GLXY22}, the existence of the effective burning velocities $\overline{H}(p)$ such that the above conclusion holds can be extended to more general two dimensional incompressible flows.


Owing to Theorem \ref{thm:homogenizationofgeometricequations}, homogenization of the curvature $G$-equation with the two dimensional cellular flow (as well as general two dimensional incompressible flows) follows from the fact that this uniform flatness result, Theorem \ref{thm:effectiveburningvelocity}, implies the existence of correctors with $\overline{F}(p)=\overline{H}(p)$ in \eqref{eq:cellproblem}, which is proved in Proposition \ref{prop:correctorofgeqn}. It is worth mentioning
here that a priori it is not clear to see that $\overline{H}(p)$ in Theorem \ref{thm:effectiveburningvelocity} coincides with the associated effective operator $\overline{F}(p)$.

\begin{corollary}\label{cor:homogenizationofgeqn}
For $\varepsilon>0$ and a uniformly continuous function $u_0$ on $\mathbb{R}^2$, let $u^{\varepsilon}$ be the unique viscosity solution to \eqref{eq:epsilong}
with $u^{\varepsilon}(\cdot,0)=u_0$. Let $\overline{H}=\overline{H}(p)$ be as in Theorem \ref{thm:effectiveburningvelocity} with $\overline{H}(0):=0$. Then, as $\varepsilon\to0$, the solution $u^{\varepsilon}$ converges locally uniformly on $\mathbb{R}^2\times[0,\infty)$ to the unique viscosity solution $u$ to
\begin{align}\label{eq:heffective}
\begin{cases}
u_t+\overline{H}(Du)=0, \quad &\text{in}\quad\mathbb{R}^2\times(0,\infty),\\
u(\cdot,0)=u_0, \quad &\text{on}\quad\mathbb{R}^2.
\end{cases}
\end{align}
\end{corollary}


We have stated the qualitative conclusion of homogenization deriving only from the the solvability of the cell problems, which is naturally expected but technically highly nontrivial. For mean curvature motions with a coercive forcing term satisfying \eqref{assumption:coercivity}, their perturbed forces also satisfy \eqref{assumption:coercivity} as stated in Proposition \ref{prop:correctors}, which we cannot expect in general (the curvature $G$-equation, for instance). We utilize this fact to obtain a rate of periodic homogenization of mean curvature motions with a coercive forcing term. For a Lipschitz function $f:\mathbb{R}^n\to\mathbb{R}^m$, we let $\|f\|_{C^{0,1}(\mathbb{R}^n)}:=\|f\|_{L^{\infty}(\mathbb{R}^n)}+\|Df\|_{L^{\infty}(\mathbb{R}^n)}$, where $Df$ is the Jacobian. Let $\hat{p}=\frac{p}{|p|}$ for $p\in\mathbb{R}^n\setminus\{0\}$.

\medskip

\begin{theorem}\label{thm:rateofforcedmcf}
Let
\begin{align*}
F(X,p,y)=-\mathrm{tr}\left\{\left(I_n-\widehat{p}\otimes\widehat{p}\right)X\right\}-c(y)|p|,
\end{align*}
for $(X,p,y)\in S^n\times\left(\mathbb{R}^n\setminus\{0\}\right)\times\mathbb{R}^n$, and suppose that a forcing term $c$ is Lipschitz continuous, $\mathbb{Z}^n$-periodic and satisfies \eqref{assumption:coercivity}. Let $u_0$ be a function on $\mathbb{R}^n$ with $\|Du_0\|_{C^{0,1}(\mathbb{R}^n)}<\infty$. For $\varepsilon>0$ let $u^{\varepsilon}$ denote the unique viscosity solution to \eqref{eq:uepsilon}, and $u$ denote the unique viscosity solution to \eqref{eq:ueffective}. Then, there exists a constant $C>0$ depending only on $n,\|c\|_{C^{0,1}\left(\mathbb{R}^n\right)},\|Du_0\|_{C^{0,1}\left(\mathbb{R}^n\right)},\delta$ such that
\begin{align*}
\|u^{\varepsilon}-u\|_{L^{\infty}\left(\mathbb{R}^n\times[0,T]\right)}\leq C(1+T)\varepsilon^{1/8}
\end{align*}
for all $T>0$, $\varepsilon\in(0,1)$.
\end{theorem}

The optimal rate is of natural interest, and as far as the author knows, the optimal rate has not been established even under the coercivity condition \eqref{assumption:coercivity}.

The following example shows that the optimal rate is slower than $O(\varepsilon)$. The forcing term is a constant in Proposition \ref{prop:example} (as well as in Proposition \ref{prop:example2}), and this is a vanishing viscosity problem rather than a homogenization problem as there is no oscillation effect. The slowing effect in Proposition \ref{prop:example} comes purely from the curvature as well.

\medskip

\begin{proposition}\label{prop:example}
Let $n=2$, $c(x)\equiv1$, and let $u_0(x)=-|x|$. Let $F,u^{\varepsilon},u$ be as in the statement of Theorem \ref{thm:rateofforcedmcf}. Then, for any $\varepsilon>0$ and for any $(x,t)\in\mathbb{R}^2\times[0,\infty)$ with $|x|=t>\varepsilon(1+e^{-1})$, we have
\begin{align}
\left|u^{\varepsilon}(x,t)-u(x,t)\right|\geq \frac{1}{2}\varepsilon\left(\log\left(\frac{t}{\varepsilon}-1\right)+1\right).
\end{align}
\end{proposition}

In the next, we present examples of traveling graphs with prescribed asymptotics when a forcing term is a positive constant, as they demonstrate that homogenization rate is related to the stability of the traveling waves if we start with 1-positively homogeneous initial data. The traveling graphs are known as the ``V-shaped traveling fronts" \cite{NT05} in 2-d, and they are studied in \cite{MRR-M13} in arbitrary spatial dimensions.

\medskip

\begin{proposition}\label{prop:example2}
Let $n\geq2,\ \alpha\in\left(0,\frac{\pi}{2}\right],\ c(x)\equiv1$. Let $A$ be a nonempty finite subset of the sphere $\mathbb{S}^{n-2}$, and for $\nu\in A$, let $(\nu,0)$ denote the unit vector in $\mathbb{R}^n$ whose first $(n-1)$-component is $\nu$ and last component is 0. Let $$u_0(x)=\sup_{\nu\in A}\left\{(\cot \alpha) x\cdot(\nu,0)\right\}.$$ Let $F,u^{\varepsilon},u$ be as in the statement of Theorem \ref{thm:rateofforcedmcf}. Then, there exists a constant $C>0$ depending only on $\alpha,|A|$ such that $$\|u^{\varepsilon}-u\|_{L^{\infty}(\mathbb{R}^{n}\times[0,\infty))}\leq C\varepsilon$$ for all $\varepsilon>0$.
\end{proposition}

We can ask whether the rate $O(\varepsilon)$ is optimal or not in laminated media when a forcing term satisfies \eqref{assumption:coercivity}. A contrasting case in general media is the example in Proposition \ref{prop:example}.

\subsection*{Organization of the paper}
We prove Theorem \ref{thm:homogenizationofgeometricequations} and Corollary \ref{cor:homogenizationofgeqn} in Section \ref{sec:geqn}, and we prove Theorem \ref{thm:rateofforcedmcf} and Propositions \ref{prop:example}, \ref{prop:example2} in Section \ref{sec:forcedmcf}.

\subsection*{Notations and conventions}
For each $n\geq1$, we set and use the following notations throughout the paper.
\begin{itemize}
    \item[$\cdot$] $x_+=\max\left\{x,0\right\}$ for $x\in\mathbb{R}$.
    \item[$\cdot$] $\hat{p}=\frac{p}{|p|}$ for $p\in\mathbb{R}^n\setminus\{0\}$.
    \item[$\cdot$] $\langle p\rangle=\sqrt{1+|p|^2}$ for $p\in\mathbb{R}^n$.
    \item[$\cdot$] $S^n$ : the set of $n\times n$ symmetric matrices, for each $n\geq1$.
    \item[$\cdot$] $I_n$ : the $n\times n$ identity matrix.
    \item[$\cdot$] $p\otimes p$ : the matrix $\left(p^ip^j\right)^n_{i,j=1}$ for $p=(p^1,\cdots,p^n)\in\mathbb{R}^n$.
    \item[$\cdot$] $\text{tr}\left\{A\right\}$ : the trace of a square matrix $A$.
    \item[$\cdot$]
    $\|A\|=\sup_{v\in\mathbb{R}^n:|v|=1}|(Av)\cdot v)|$ for each $n\times n$ matrix $A$.
    \item[$\cdot$] $B_r(x)=\left\{y\in\mathbb{R}^n:|y-x|<r\right\}$ for $x\in\mathbb{R}^n,\ r>0$.
    \item[$\cdot$] $Q_r(P)=B_r(x)\times\left((t-r,t+r)\cap[0,\infty)\right)$ for $P=(x,t)\in\mathbb{R}^n\times[0,\infty),\ r>0$.
    \item[$\cdot$] $\overline{Q}_r(P)$ : the closure of $Q_r(P)$ in $\mathbb{R}^n\times[0,\infty)$ for $P\in\mathbb{R}^n\times[0,\infty),\ r>0$.
    \item[$\cdot$] $\|f\|_{C^{0,1}(\mathbb{R}^n)}=\|f\|_{L^{\infty}(\mathbb{R}^n)}+\|Df\|_{L^{\infty}(\mathbb{R}^n)}$ for a Lipschitz function $f:\mathbb{R}^n\to\mathbb{R}^m$ for $m\geq1$, where $Df$ denotes the Jacobian. 
    \item[$\cdot$] $F^*$ and $F_*$ : the upper and lower-semicontinuous envelope of $F:S^n\times (\mathbb{R}^n\setminus\{0\})\times \mathbb{R}^n\to\mathbb{R}$, respectively.
\end{itemize}

We follow the convention throughout the paper that a number $C=C(\cdot)>0$ denotes a positive constant that may vary line by line, and that its dependency on parameters (such as, $\varepsilon,\eta,\mu,r,\cdots$) is specified in its arguments. Specifying the dependency in the arguments is also applied to various parameters that appear in this paper, not just to $C>0$.

\section{Proof of Theorem \ref{thm:homogenizationofgeometricequations}}\label{sec:geqn}
This section is mainly devoted to the proof of Theorem \ref{thm:homogenizationofgeometricequations}.

Let $F:S^n\times\left(\mathbb{R}^n\setminus\{0\}\right)\times\mathbb{R}^n\to\mathbb{R}$ be given. Throughout this section, we let
\begin{equation}
\begin{cases}
F^{\eta}(X,p,y):=\inf_{|e|\leq\eta}F(X,p,y+e), \quad & \\
F_{\eta}(X,p,y):=\sup_{|e|\leq\eta}F(X,p,y+e) \quad & 
\end{cases}\label{line:perturbedoperators}
\end{equation}
for $(X,p,y)\in S^n\times\left(\mathbb{R}^n\setminus\{0\}\right)\times\mathbb{R}^n$ and $\eta\geq0$.

\begin{proof}[Proof of Theorem \ref{thm:homogenizationofgeometricequations}]

$\ \ $

\medskip

We first of all note that there are sub/supersolutions $u^{\mp}$ to \eqref{eq:uepsilon}, which are independent of $\varepsilon\in(0,1]$ (see \cite[Lemma 4.3.4, Theorem 4.3.1]{G06}), such that
$$
u^-\leq u^+,\quad\textrm{and}\qquad\lim_{t\to0}\sup_{x\in\mathbb{R}^n}|u^{\mp}(x,t)-u_0(x)|=0,
$$
which follows from the conditions (I), (II), (III) and (i) of (V). By Perron's method, there exists a viscosity solution $u^{\varepsilon}$ to \eqref{eq:uepsilon}. By the comparison principle for \eqref{eq:uepsilon} (see \cite[Subsection 13.2]{CM14}), we see that $u^-\leq u^{\varepsilon}\leq u^+$, and that $u^{\varepsilon}$ is a unique solution and it is continuous.

Let $\overline{u}:=\limsup_{\varepsilon\to0}^*u^{\varepsilon}$ and $\underline{u}:=\liminf_{\varepsilon\to0\ *}u^{\varepsilon}$. Then, we have that, by the definition of $\underline{u},\overline{u}$,
$$
u^-\leq\underline{u}\leq\overline{u}\leq u^+,
$$
and thus that $\underline{u}(\cdot,0)=\overline{u}(\cdot,0)=u_0$. Therefore, it suffices to prove that $\overline{u}$ is a viscosity subsolution to \eqref{eq:ueffective} in $\mathbb{R}^n\times(0,\infty)$, and that $\underline{u}$ is a viscosity supersolution to \eqref{eq:ueffective} in $\mathbb{R}^n\times(0,\infty)$. Then, the comparison principle for \eqref{eq:ueffective} implies that $\overline{u}\leq\underline{u}$, which then implies the local uniform convergence of $u^{\varepsilon}$ to $\overline{u}=\underline{u}(=:u)$ on $\mathbb{R}^n\times[0,\infty)$.

\medskip

\textbf{Claim 1: The function $\underline{u}$ is a viscosity supersolution to \eqref{eq:ueffective} in $\mathbb{R}^n\times(0,\infty)$.} 

\medskip

\textbf{Step 1.1: Introduce parameters $r,\theta>0$ from the assumption for the contrary.}

\medskip

Suppose the contrary for contradiction. Then, there exist $P_0=(x_0,t_0)\in\mathbb{R}^n\times(0,\infty)$, $r\in(0,t_0)$ and a $C^2$ function $\varphi$ in $Q_r(P_0)$ such that
\begin{equation*}
\begin{cases}
\underline{u}(P_0)=\varphi(P_0), \quad &\\
\underline{u}\geq\varphi \quad &\text{on}\quad \overline{Q}_r(P_0),\\
\varphi_t(P_0)+\overline{F}(D\varphi(P_0))=:-\theta<0.\quad&
\end{cases}
\end{equation*}
Let $p=D\varphi(P_0),\ \overline{\lambda}=\overline{F}(p),\ \lambda_t=\varphi_t(P_0)$ so that we have
\begin{align}
\overline{\lambda}+\theta=-\lambda_t.\label{line:overlinelambda}
\end{align}

By replacing $\varphi$ by $-|(x,t)-P_0|^4+\varphi$ if necessary, we can assume without loss of generality that there exists $\delta_1=\delta_1(r)>0$ such that
\begin{equation*}
\begin{cases}
\underline{u}-2\delta_1\geq\varphi \quad & \text{on}\quad \overline{Q}_r(P_0)\setminus Q_{r/2}(P_0),\\
\underline{u}>\varphi \quad & \text{on}\quad \overline{Q}_r(P_0)\setminus \{P_0\}.
\end{cases}
\end{equation*}

\medskip

We only cover the case $\overline{\lambda}=\overline{F}(p)\geq0$ in the proof of Claim 1 to avoid a lengthy paper. The other case $\overline{\lambda}=\overline{F}(p)<0$ of Claim 1, i.e., proving that $\underline{u}$ is a supersolution when $\overline{\lambda}=\overline{F}(p)<0$ is omitted, as this case corresponds to the argument of \cite[Section 8]{CM14} that shows $\underline{u}$ is a supersolution. We use perturbed approximate correctors instead of perturbed correctors in both of the cases, and how we use perturbed correctors is demonstrated in detail from now on.

\medskip

\textbf{Step 1.2: Introduce approximate correctors $w_{\lambda}$ and $w_{\lambda,2\eta}$.}

\medskip

For $\lambda>0$, we let $w_{\lambda}$ be the solution to
\begin{align*}
\lambda w_{\lambda}+F\left(D^2w_{\lambda},p+Dw_{\lambda},y\right)=0 \quad \text{on}\quad\mathbb{R}^n.
\end{align*}
By a simple comparison argument, we see that $$-\sup v-\frac{\overline{F}(p)}{\lambda}+v\leq w_{\lambda}\leq-\inf v-\frac{\overline{F}(p)}{\lambda}+v,$$
where $v$ is a $\mathbb{Z}^n$-periodic viscosity solution to \eqref{eq:cellproblem} satisfying $$\sup v-\inf v\leq\frac{1}{4}\overline{\kappa}_0$$ for some constant $\overline{\kappa}_0=\overline{\kappa}_0(p)>0$. The existence of such a number $\overline{\kappa}_0=\overline{\kappa}_0(p)>0$ is from the hypothesis of Theorem \ref{thm:homogenizationofgeometricequations}, as the upper-semicontinuous (and lower-semicontinuous) envelope of $v$ is bounded from above (below, respectively), by the definition of viscosity solutions. Accordingly, we have
\begin{align*}
\sup w_{\lambda}-\inf w_{\lambda}\leq\frac{1}{2}\overline{\kappa}_0,
\end{align*}
and there exists $\lambda=\lambda(\theta,p)\in(0,1)$ such that
\begin{align*}
\|(-\lambda w_{\lambda})-\overline{\lambda}\|_{L^{\infty}(\mathbb{R}^n)}\leq\frac{1}{16}\theta.
\end{align*}

We now consider an approximate corrector of the perturbed problem; for $\eta\geq0$, let $w_{\lambda,2\eta}$ the solution to
\begin{align}
\lambda w_{\lambda,2\eta}+F_{2\eta}\left(D^2w_{\lambda,2\eta},p+Dw_{\lambda,2\eta},y\right)=0 \qquad \text{on}\quad\mathbb{R}^n.\label{line:eqnofwlambdaeta}
\end{align}
Here, $F_{2\eta}$ is given as in \eqref{line:perturbedoperators}. Choose $\eta=\eta(\lambda,\theta,\overline{\kappa}_0,p)=\eta(\theta,p)\in(0,1)$ such that
\begin{align*}
\|w_{\lambda,2\eta}-w_{\lambda}\|_{L^{\infty}(\mathbb{R}^n)}\leq\min\left\{\frac{1}{16}\theta,\frac{1}{4}\overline{\kappa}_0\right\}
\end{align*}
so that
\begin{align}\label{line:amplitudeofcorrector}
\sup w_{\lambda,2\eta}-\inf w_{\lambda,2\eta}\leq\overline{\kappa}_0
\end{align}
and
\begin{align}
\|(-\lambda w_{\lambda,2\eta})-\overline{\lambda}\|_{L^{\infty}(\mathbb{R}^n)}\leq\frac{1}{8}\theta.\label{line:convergenceofw}
\end{align}

\medskip

\textbf{Step 1.3: Define an extension and the (sup-)convolution of the test function $\varphi$.}

\medskip

Write
\begin{align}
\varphi(x,t)=\varphi(P_0)+p\cdot(x-x_0)+\lambda_t(t-t_0)+\psi(x,t)\label{line:psi}
\end{align}
with
\begin{align}
|D\psi|,|\psi_t|\leq\mu \quad \text{on}\quad \overline{Q}_r(P_0)\label{line:psiderivative}
\end{align}
for some $\mu=\mu(r)>0$ that goes to 0 as $r\to0$. We extend $\psi$ to a $C^2$ function on $\mathbb{R}^n\times[0,\infty)$, still denoted by $\psi$, that satisfies
\begin{align}
|D\psi|+\|D^2\psi\|\leq\mu_0 \quad \text{on }\quad \mathbb{R}^n\times[0,\infty)\label{line:mu0}
\end{align}
for some $\mu_0\leq1$ by replacing $r>0$ by a smaller number if necessary. We also keep the notation for $\varphi$.

\medskip

Let
\begin{equation}\label{line:convolution}
\begin{cases}
\widetilde{\varphi}^{\varepsilon}(x,t):=\varphi(x,t)+\varepsilon\left(w_{\lambda.2\eta}\left(\frac{x}{\varepsilon}\right)-w_{\lambda,2\eta}(0)\right),\\
\overline{\varphi}^{\varepsilon}(x,t):=\sup_{z\in\overline{B}_{\varepsilon\eta}(x)}\widetilde{\varphi}^{\varepsilon}(z,t),\\
\varphi^{\varepsilon}(x,t):=\sup_{z\in\mathbb{R}^n}\left(\overline{\varphi}^{\varepsilon}(z,t)-\frac{|x-z|^4}{4\varepsilon^3\rho}\right),
\end{cases}
\end{equation}
where $\rho\in(0,1)$ is to be determined later. By \cite[Lemma 13.2.(B)]{CM14}, there exists $\varepsilon=\varepsilon(\delta_1,\overline{\kappa}_0,|p|+\mu_0)=\varepsilon(r,p)\in(0,1)$ such that
\begin{align*}
u^{\varepsilon}-\delta_1\geq\varphi^{\varepsilon} \quad \text{on}\quad \overline{Q}_r(P_0)\setminus Q_{r/2}(P_0), 
\end{align*}
and that the infimum of $u^{\varepsilon}-\varphi^{\varepsilon}$ on $\overline{Q}_r(P_0)$ is attained in $Q_{r/2}(P_0)$, say at $P_{\varepsilon}=(x_{\varepsilon},t_{\varepsilon})\in Q_{r/2}(P_0)$. By \cite[Lemma 13.2.(B)]{CM14}, and also by \eqref{line:amplitudeofcorrector}, there exist $\rho=\rho(\eta,\overline{\kappa}_0,|p|+\mu_0)=\rho(\theta,p)\in(0,1)$, $\overline{x}_{\varepsilon}\in\mathbb{R}^n$ such that
\begin{align}
\varphi^{\varepsilon}(x_{\varepsilon},t_{\varepsilon})=\overline{\varphi}^{\varepsilon}(\overline{x}_{\varepsilon},t_{\varepsilon})-\frac{|x_{\varepsilon}-\overline{x}_{\varepsilon}|^4}{4\varepsilon^3\rho}\quad\text{and}\quad|\overline{x}_{\varepsilon}-x_{\varepsilon}|\leq\varepsilon\eta.\label{line:overlinex}
\end{align}

\medskip

\textbf{Step 1.4: Obtain the viscosity inequalities from the Crandall-Ishii lemma}

\medskip

Unraveling the infimum in \eqref{line:convolution} and by the choice of $\overline{x}_{\varepsilon}\in\mathbb{R}^n$, we see that
\begin{align*}
(x,(y,t))\in\mathbb{R}^n\times\overline{Q}_r(P_0)\longmapsto\overline{\varphi}^{\varepsilon}(x,t)-u^{\varepsilon}(y,t)-\frac{|x-y|^4}{4\varepsilon^3\rho}
\end{align*}
attains a maximum at $(\overline{x}_{\varepsilon},(x_{\varepsilon},t_{\varepsilon}))\in\mathbb{R}^n\times Q_{r/2}(P_0)$. Since \cite[(8.5)]{CIL92} holds for our $F$ (from the condition (i) of (V)) and for $\overline{\varphi}^{\varepsilon},u^{\varepsilon}$, we can apply the Crandall-Ishii lemma \cite[Theorem 8.3]{CIL92} to see that for every $\gamma>0$, there exist $X,Y\in S^n$ such that
\begin{equation}
\begin{cases}
(b_1,q,X)\in\overline{\mathcal{P}}^{2,+}\overline{\varphi}^{\varepsilon}(\overline{x}_{\varepsilon},t_{\varepsilon}), \quad &\\
(b_2,q,Y)\in\overline{\mathcal{P}}^{2,-}u^{\varepsilon}(x_{\varepsilon},t_{\varepsilon}), \quad &\\
b_1-b_2=0=\Phi_t(\overline{x}_{\varepsilon},x_{\varepsilon},t_{\varepsilon}), \quad &\\
-\left(\frac{1}{\gamma}+\|A\|\right)I_{2n}\leq\begin{pmatrix}X &0 \\0 &-Y \end{pmatrix}\leq A+\gamma A^2,
\end{cases}\label{line:semijets}
\end{equation}
where $\Phi(x,y,t):=\frac{|x-y|^4}{4\varepsilon^3\rho}$ and
\begin{align*}
q&:=D_x\Phi(\overline{x}_{\varepsilon},x_{\varepsilon},t_{\varepsilon})=-D_y\Phi(\overline{x}_{\varepsilon},x_{\varepsilon},t_{\varepsilon})=\delta(\overline{x}_{\varepsilon}-x_{\varepsilon})\quad\text{with}\quad\delta:=\frac{|x_{\varepsilon}-\overline{x}_{\varepsilon}|^2}{\varepsilon^3\rho},\\
A:&=D^2_{(x,y)}\Phi(\overline{x}_{\varepsilon},x_{\varepsilon},t_{\varepsilon})=\delta\begin{pmatrix}I_n+2\widehat{q}\otimes\widehat{q} &-I_n-2\widehat{q}\otimes\widehat{q} \\-I_n-2\widehat{q}\otimes\widehat{q} &I_n+2\widehat{q}\otimes\widehat{q} \end{pmatrix}\ \text{with}\  \widehat{q}:=\frac{q}{|q|}\ (\text{if }q\neq0).
\end{align*}

\medskip

By \cite[Lemma 13.1.(ii)]{CM14}, there exists $\widetilde{x}_{\varepsilon}\in\mathbb{R}^n$ such that
\begin{align}
(b_1,q,X)\in\overline{\mathcal{P}}^{2,+}\widetilde{\varphi}^{\varepsilon}(\widetilde{x}_{\varepsilon},t_{\varepsilon}),\quad\overline{\varphi}^{\varepsilon}(\overline{x}_{\varepsilon},t_{\varepsilon})=\widetilde{\varphi}^{\varepsilon}(\widetilde{x}_{\varepsilon},t_{\varepsilon}),\quad|\widetilde{x}_{\varepsilon}-\overline{x}_{\varepsilon}|\leq\varepsilon\eta.\label{line:xtilde}
\end{align}
Note that $v_{\lambda,2\eta}(x):=\varepsilon\left(w_{\lambda,2\eta}\left(\frac{x}{\varepsilon}\right)-w_{\lambda,2\eta}(0)\right)$ is a viscosity solution to
\begin{align*}
F_{2\eta}\left(\varepsilon D^2v_{\lambda,2\eta}(x),p+Dv_{\lambda,2\eta}(x),\frac{x}{\varepsilon}\right)\leq\overline{\lambda}+\frac{1}{8}\theta\quad\text{on}\quad\mathbb{R}^n,
\end{align*}
from \eqref{line:eqnofwlambdaeta}, \eqref{line:convergenceofw}. Therefore, from \eqref{line:xtilde} and $\widetilde{\varphi}^{\varepsilon}=\varphi+v_{\lambda,2\eta}$, we have
\begin{align}
b_1=\varphi_t(\widetilde{P}_{\varepsilon}),\ (q-D\varphi(\widetilde{P}_{\varepsilon}),X-D^2\varphi(\widetilde{P}_{\varepsilon}))\in\overline{\mathcal{J}}^{2,+}v_{\lambda,2\eta}(\widetilde{x}_{\varepsilon})\ \text{with}\ \widetilde{P}_{\varepsilon}=(\widetilde{x}_{\varepsilon},t_{\varepsilon})\label{line:semijetofv}
\end{align}
and
\begin{align*}
\left(F_{2\eta}\right)_*\left(\varepsilon X-\varepsilon D^2\varphi(\widetilde{P}_{\varepsilon}),p+q-D\varphi(\widetilde{P}_{\varepsilon}),\frac{\widetilde{x}_{\varepsilon}}{\varepsilon}\right)\leq\overline{\lambda}+\frac{1}{8}\theta.
\end{align*}
Finally, by \eqref{line:psi} and $(b_2,q,Y)\in\overline{\mathcal{P}}^{2,-}u^{\varepsilon}(x_{\varepsilon},t_{\varepsilon})$, we obtain the viscosity inequalities
\begin{equation}
\begin{cases}
\left(F_{2\eta}\right)_*\left(\varepsilon X-\varepsilon D^2\psi(\widetilde{P}_{\varepsilon}),q-D\psi(\widetilde{P}_{\varepsilon}),\frac{\widetilde{x}_{\varepsilon}}{\varepsilon}\right)\leq\overline{\lambda}+\frac{1}{8}\theta, \quad &\\
b_2+F^*\left(\varepsilon Y,q,\frac{x_{\varepsilon}}{\varepsilon}\right)\geq0. \quad &
\end{cases}\label{line:viscosityinequalities}
\end{equation}

\medskip

\textbf{Step 1.5: Obtain bounds of the gradient $q$ and of the Hessians $X,Y$.}

\medskip

We can check that $\|A\|=6\delta$ and
\begin{align*}
A^2\leq 18\delta^2 \begin{pmatrix}I_n &-I_n \\-I_n &I_n \end{pmatrix}.
\end{align*}
Setting $\gamma=\frac{1}{3\delta}$ in \eqref{line:semijets}, we obtain
\begin{align*}
-9\delta I_n\leq X\leq Y\leq 9\delta I_n
\end{align*}
and in turn, with $\delta\leq\frac{\eta^2}{\varepsilon\rho}$ (from \eqref{line:overlinex}),
\begin{align}
|q|\leq\frac{\eta^3}{\rho}\quad\text{and}\quad-\frac{9\eta^2}{\rho}I_n\leq \varepsilon X\leq \varepsilon Y\leq \frac{9\eta^2}{\rho}I_n.\label{line:boundfromabove}
\end{align}

\medskip

Now, as a crucial step, we separate the gradient $q$ from the origin with a constant that depends only on $\theta,p$. Choose $r=r(\theta)\in(0,1)$ small, $\varepsilon=\varepsilon(r,p)=\varepsilon(\theta,p)\in(0,1)$ small enough that $\mu=\mu(r)\leq\frac{\theta}{4}$ where $\mu>0$ is given in \eqref{line:psiderivative}, and that
\begin{align*}
P_\varepsilon=(x_{\varepsilon},t_{\varepsilon}),\quad|\overline{x}_{\varepsilon}-x_{\varepsilon}|\leq\varepsilon\eta,\quad|\widetilde{x}_{\varepsilon}-\overline{x}_{\varepsilon}|\leq\varepsilon\eta
\end{align*}
imply $\widetilde{P}_{\varepsilon}=(\widetilde{x}_{\varepsilon},t_{\varepsilon})\in Q_r(P_0)$, and thus that
\begin{align}
|\psi_{t}(\widetilde{P}_{\varepsilon})|\leq\mu\leq\frac{\theta}{4}.\label{line:psi_t}
\end{align}
From \eqref{line:psi}, \eqref{line:semijets}, \eqref{line:semijetofv}, we have $b_1=b_2=\lambda_t+\psi_t(\widetilde{P}_{\varepsilon})$. Using the assumption $\overline{\lambda}=\overline{F}(p)\geq0$, we link the identities/inequalities \eqref{line:overlinelambda}, \eqref{line:semijets}, \eqref{line:viscosityinequalities}, \eqref{line:boundfromabove}, \eqref{line:psi_t} to obtain a lower bound as follows:
\begin{align*}
\frac{\theta}{2}\leq\overline{\lambda}+\theta-\psi_t(\widetilde{P}_{\varepsilon})=-\lambda_t-\psi_t(\widetilde{P}_{\varepsilon})&=-b_2\\
&\leq F^*\left(\varepsilon Y,q,\frac{x_{\varepsilon}}{\varepsilon}\right)\leq F^*\left(\varepsilon X, q, \frac{x_{\varepsilon}}{\varepsilon}\right).
\end{align*}
Moreover, the function $\overline{\varphi}^{\varepsilon}$ is sup-convoluted by the definition \eqref{line:convolution}, and therefore, its semijet $(b_1,q,X)\in\overline{\varphi}^{\varepsilon}$ enjoys a bound with the geometric operator $F$ as follows (\cite[Lemma 13.1.(ii)]{CM14}):
\begin{align*}
\frac{\theta}{2}\leq F^*\left(\varepsilon X, q, \frac{x_{\varepsilon}}{\varepsilon}\right)\leq c|q|
\end{align*}
for some $c=c(\eta)=c(\theta,p)>0$. Therefore, we obtain
\begin{align}
\frac{\theta}{2c}\leq|q|.\label{line:separatefromtheorigin}
\end{align}

\medskip

\textbf{Step 1.6: Derive a contradiction.}

\medskip

Note that the operator $F$ is uniformly continuous on
\begin{align*}
\left\{(X',p',y')\in S^n\times\mathbb{R}^n\times\mathbb{R}^n:\|X'\|\leq1+\frac{9\eta^2}{\rho},\ \frac{\theta}{4c}\leq|p'|\leq1+\frac{\eta^3}{\rho}\right\}.
\end{align*}
By this fact, we can choose $r=r(\theta,p)\in(0,1),$ $\varepsilon=\varepsilon(r,\theta,p)\in(0,1)$ such that
\begin{align*}
\varepsilon\|D^2\psi(\widetilde{P}_{\varepsilon})\|,|D\psi(\widetilde{P}_{\varepsilon})|\leq\mu(r)\leq\min\left\{1,\frac{\theta}{2},\frac{\theta}{4c}\right\},
\end{align*}
and thus that, by \eqref{line:boundfromabove}, \eqref{line:separatefromtheorigin},
\begin{align*}
\|\varepsilon X\|,\|\varepsilon X-\varepsilon D^2\psi(\widetilde{P}_{\varepsilon})\|\leq1+\frac{9\eta^2}{\rho}\ \ \text{and}\ \ \frac{\theta}{4c}\leq|q|,|q-D\psi(\widetilde{P}_{\varepsilon})|\leq1+\frac{\eta^3}{\rho},
\end{align*}
and finally that
\begin{align}
F\left(\varepsilon X-\varepsilon D^2\psi(\widetilde{P}_{\varepsilon}),q-D\psi(\widetilde{P}_{\varepsilon}),\frac{x_{\varepsilon}}{\varepsilon}\right)\geq-\frac{1}{4}\theta+ F\left(\varepsilon X,q,\frac{x_{\varepsilon}}{\varepsilon}\right).\label{line:uniformcontinuity}
\end{align}
We derive a contradiction by linking the inequalities \eqref{line:overlinelambda}, \eqref{line:viscosityinequalities}, \eqref{line:boundfromabove}, \eqref{line:psi_t}, \eqref{line:uniformcontinuity}:
\begin{align*}
\overline{\lambda}+\frac{1}{8}\theta&\geq\left(F_{2\eta}\right)_*\left(\varepsilon X-\varepsilon D^2\psi(\widetilde{P}_{\varepsilon}),q-D\psi(\widetilde{P}_{\varepsilon}),\frac{\widetilde{x}_{\varepsilon}}{\varepsilon}\right)\\
&=F_{2\eta}\left(\varepsilon X-\varepsilon D^2\psi(\widetilde{P}_{\varepsilon}),q-D\psi(\widetilde{P}_{\varepsilon}),\frac{\widetilde{x}_{\varepsilon}}{\varepsilon}\right)\\
&\geq F\left(\varepsilon X-\varepsilon D^2\psi(\widetilde{P}_{\varepsilon}),q-D\psi(\widetilde{P}_{\varepsilon}),\frac{x_{\varepsilon}}{\varepsilon}\right)\\
&\geq-\frac{1}{4}\theta+ F\left(\varepsilon X,q,\frac{x_{\varepsilon}}{\varepsilon}\right)\\
&\geq-\frac{1}{4}\theta+ F\left(\varepsilon Y,q,\frac{x_{\varepsilon}}{\varepsilon}\right)\\
&\geq-\frac{1}{4}\theta-b_2\\
&=-\frac{1}{4}\theta-\lambda_t-\psi_t(\widetilde{P}_{\varepsilon})\\
&\geq\overline{\lambda}+\frac{1}{2}\theta,
\end{align*}
which completes the proof of Claim 1 in the case $\overline{\lambda}=\overline{F}(p)\geq0$.

\medskip

\textbf{Claim 2: The function $\overline{u}$ is a viscosity subsolution to \eqref{eq:ueffective} in $\mathbb{R}^n\times(0,\infty)$.} 

\medskip

\textbf{Step 2.1: Introduce parameters $r,\theta>0$ from the assumption for the contrary.}

\medskip

Suppose the contrary for contradiction. Then, there exist $P_0=(x_0,t_0)\in\mathbb{R}^n\times(0,\infty)$, $r\in(0,t_0)$ and a $C^2$ function $\varphi$ in $Q_r(P_0)$ such that
\begin{equation*}
\begin{cases}
\overline{u}(P_0)=\varphi(P_0), \quad &\\
\overline{u}\leq\varphi \quad &\text{on}\quad \overline{Q}_r(P_0),\\
\varphi_t(P_0)+\overline{F}(D\varphi(P_0))=:\theta>0.\quad&
\end{cases}
\end{equation*}
Let $p=D\varphi(P_0),\ \overline{\lambda}=\overline{F}(p),\ \lambda_t=\varphi_t(P_0)$ so that we have
\begin{align}
\lambda_t+\overline{\lambda}=\theta.\label{line:overlinelambda2}
\end{align}

By replacing $\varphi$ by $|(x,t)-P_0|^4+\varphi$ if necessary, we can assume without loss of generality that there exists $\delta_1=\delta_1(r)>0$ such that
\begin{equation*}
\begin{cases}
\overline{u}+2\delta_1\leq\varphi \quad & \text{on}\quad \overline{Q}_r(P_0)\setminus Q_{r/2}(P_0),\\
\overline{u}<\varphi \quad & \text{on}\quad \overline{Q}_r(P_0)\setminus \{P_0\}.
\end{cases}
\end{equation*}

\medskip

We only handle the case $\overline{\lambda}=\overline{F}(p)>0$ in the proof of Claim 2. The other case $\overline{\lambda}=\overline{F}(p)\leq0$ of Claim 2, i.e., proving that $\overline{u}$ is a subsolution when $\overline{\lambda}=\overline{F}(p)\leq0$ is omitted, as this case corresponds to the argument of \cite[Section 8]{CM14} that shows $\overline{u}$ is a subsolution. We instead explain the use of perturbed approximate correctors in detail in the below (but in the opposite direction of perturbation to the proof of Claim 1).

\medskip

\textbf{Step 2.2: Introduce approximate correctors $w_{\lambda}$ and $w_{\lambda}^{2\eta}$.}

\medskip

Let $w_{\lambda}$ be the solution to
\begin{align*}
\lambda w_{\lambda}+F\left(D^2w_{\lambda},p+Dw_{\lambda},y\right)=0 \quad \text{on}\quad \mathbb{R}^n.
\end{align*}
Similarly as in Step 1.1.1, namely by comparing $w_{\lambda}$ with $v$ (with additional constants), we see that there exist $\overline{\kappa}_0=\overline{\kappa}_0(p)>0$ and $\lambda=\lambda(\theta,p)\in(0,1)$ such that
\begin{align*}
\sup w_{\lambda}-\inf w_{\lambda}\leq\frac{1}{2}\overline{\kappa}_0
\end{align*}
and
\begin{align*}
\|(-\lambda w_{\lambda})-\overline{\lambda}\|_{L^{\infty}(\mathbb{R}^n)}\leq\min\left\{\frac{1}{16}\theta,\frac{1}{16}\overline{\lambda}\right\}.
\end{align*}
We consider an approximate corrector of the perturbed problem; for $\eta\geq0$, let $w_{\lambda}^{2\eta}$ the solution to
\begin{align}
\lambda w_{\lambda}^{2\eta}+F^{2\eta}\left(D^2w_{\lambda}^{2\eta},p+Dw_{\lambda}^{2\eta},y\right)=0 \quad \text{on}\quad \mathbb{R}^n.\label{line:eqnofwlambdaeta2}
\end{align}
Here, $F^{2\eta}$ is given as in \eqref{line:perturbedoperators}. Choose $\eta=\eta(\lambda,\theta,\overline{\kappa}_0,p)=\eta(\theta,p)\in(0,1)$ such that
\begin{align*}
\|w_{\lambda}^{2\eta}-w_{\lambda}\|_{L^{\infty}(\mathbb{R}^n)}\leq\min\left\{\frac{1}{16}\theta,\frac{1}{16}\overline{\lambda},\frac{1}{4}\overline{\kappa}_0\right\}
\end{align*}
so that
\begin{align}\label{line:amplitudeofcorrector2}
\sup w_{\lambda}^{2\eta}-\inf w_{\lambda}^{2\eta}\leq\overline{\kappa}_0
\end{align}
and
\begin{align}
\|(-\lambda w_{\lambda}^{2\eta})-\overline{\lambda}\|_{L^{\infty}(\mathbb{R}^n)}\leq\min\left\{\frac{1}{8}\theta,\frac{1}{8}\overline{\lambda}\right\}.\label{line:convergenceofw2}
\end{align}

\medskip

\textbf{Step 2.3: Define an extension of the test function $\varphi$, a linear functional $\ell$ and the (sup-)convolution of $u^{\varepsilon}$.}

\medskip

Let $\psi$ be a $C^2$ function on $\mathbb{R}^n\times[0,\infty)$ defined as in \eqref{line:psi} (with abuse of notations for the extension) satisfying \eqref{line:psiderivative}, \eqref{line:mu0}. For $(x,t)\in\mathbb{R}^n\times[0,\infty)$, let 
\begin{equation}\label{line:ell}
\begin{cases}
\ell(x,t):=\varphi(P_0)+p\cdot(x-x_0)+\left(-\lambda w_{\lambda}^{2\eta}(0)\right)(t-t_0),\\
\widetilde{\ell}^{\varepsilon}(x,t):=\ell(x,t)+\varepsilon\left(w_{\lambda}^{2\eta}\left(\frac{x}{\varepsilon}\right)-w_{\lambda}^{2\eta}(0)\right),\\
\widehat{\ell}^{\varepsilon}(x,t):=\inf_{z\in\mathbb{R}^n}\left(\widetilde{\ell}^{\varepsilon}(z,t)+\frac{|x-z|^4}{4\varepsilon^3\rho}\right),
\end{cases}
\end{equation}
where $\rho\in(0,1)$ is to be determined later, and
\begin{equation}\label{line:uhat}
\begin{cases}
\overline{u}^{\varepsilon}(x,t):=u^{\varepsilon}(x,t)+\left((-\lambda w_{\lambda}^{2\eta}(0))-\lambda_t\right)(t-t_0)-\psi(x,t),\\
\widehat{u}^{\varepsilon}(x,t):=\sup_{z\in\overline{B}_{\varepsilon\eta}(x)}\overline{u}^{\varepsilon}(z,t).
\end{cases}
\end{equation}

For $\varepsilon=\varepsilon(\delta_1)=\varepsilon(r)\in(0,1)$ small enough,
\begin{align*}
u^{\varepsilon}+\delta_1\leq\varphi \quad \text{on}\quad \overline{Q}_r(P_0)\setminus Q_{r/2}(P_0), 
\end{align*}
which in turn implies, by \eqref{line:amplitudeofcorrector2}, \eqref{line:ell}, \eqref{line:uhat},
\begin{align*}
\overline{u}^{\varepsilon}+\delta_1\leq\widetilde{\ell}^{\varepsilon}+\varepsilon\overline{\kappa}_0 \quad \text{on}\quad \overline{Q}_r(P_0)\setminus Q_{r/2}(P_0). 
\end{align*}
By the fact that $\limsup_{\varepsilon}^*\widehat{u}^{\varepsilon}=\limsup_{\varepsilon}^*\overline{u}^{\varepsilon}$ and by \cite[Lemma 13.2.(A).(vii)]{CM14}, we see that there exists $\varepsilon=\varepsilon(P_0,r,\delta_1,\overline{\kappa}_0,|p|)=\varepsilon(P_0,r,p)\in(0,1)$ small enough such that
\begin{align*}
\widehat{u}^{\varepsilon}+\frac{1}{2}\delta_1\leq\widehat{\ell}^{\varepsilon}\quad \text{on}\quad \overline{Q}_r(P_0)\setminus Q_{r/2}(P_0), 
\end{align*}
and that the supremum of $\widehat{u}^{\varepsilon}-\widehat{\ell}^{\varepsilon}$ on $\overline{Q}_r(P_0)$ is attained in $Q_{r/2}(P_0)$, say at $\widehat{P}_{\varepsilon}=(\widehat{x}_{\varepsilon},t_{\varepsilon})\in Q_{r/2}(P_0).$ Also, by \cite[Lemma 13.2.(A).(vii)]{CM14}, there exist $\rho=\rho(\eta,\overline{\kappa}_0,|p|)=\rho(\theta,p)\in(0,1)$, $\widetilde{x}_{\varepsilon}\in\mathbb{R}^n$ such that
\begin{align}
\widehat{\ell}^{\varepsilon}(\widehat{x}_{\varepsilon},t_{\varepsilon})=\widetilde{\ell}^{\varepsilon}(\widetilde{x}_{\varepsilon},t_{\varepsilon})+\frac{|\widehat{x}_{\varepsilon}-\widetilde{x}_{\varepsilon}|^4}{4\varepsilon^3\rho}\quad\text{and}\quad|\widehat{x}_{\varepsilon}-\widetilde{x}_{\varepsilon}|\leq\varepsilon\eta.\label{line:widetildex}
\end{align}

\medskip

\textbf{Step 2.4: Obtain the viscosity inequalities from the Crandall-Ishii lemma.}

\medskip

Unraveling the infimum in \eqref{line:ell}, the supremum in \eqref{line:uhat} and by the choice of $\widetilde{x}_{\varepsilon}\in\mathbb{R}^n$, we see that
\begin{align*}
(x,y,t)\in\overline{B}_r(x_0)\times\mathbb{R}^n\times[t_0-r,t_0+r]\longmapsto\widehat{u}^{\varepsilon}(x,t)-\widetilde{\ell}^{\varepsilon}(y,t)-\frac{|x-y|^4}{4\varepsilon^3\rho}
\end{align*}
attains a maximum at $(\widehat{x}_{\varepsilon},\widetilde{x}_{\varepsilon},t_{\varepsilon})\in B_{r/2}(x_0)\times\mathbb{R}^n\times\left(t_0-\frac{1}{2}r,t_0+\frac{1}{2}r\right)$. Since \cite[(8.5)]{CIL92} holds for our $F$ (by the condition (i) of (V)) and for $\widehat{u}^{\varepsilon},\widetilde{\ell}^{\varepsilon}$ (with the aid of \cite[Lemma 13.1.(ii)]{CM14} for $\widehat{u}^{\varepsilon}$), we can apply the Crandall-Ishii lemma \cite[Theorem 8.3]{CIL92} to see that for every $\gamma>0$, there exist $X,Y\in S^n$ such that
\begin{equation}
\begin{cases}
(b_1,q,X)\in\overline{\mathcal{P}}^{2,+}\widehat{u}^{\varepsilon}(\widehat{x}_{\varepsilon},t_{\varepsilon}), \quad &\\
(b_2,q,Y)\in\overline{\mathcal{P}}^{2,-}\widetilde{\ell}^{\varepsilon}(\widetilde{x}_{\varepsilon},t_{\varepsilon}), \quad &\\
b_1-b_2=0=\Phi_t(\overline{x}_{\varepsilon},x_{\varepsilon},t_{\varepsilon}), \quad &\\
-\left(\frac{1}{\gamma}+\|A\|\right)I_{2n}\leq\begin{pmatrix}X &0 \\0 &-Y \end{pmatrix}\leq A+\gamma A^2,
\end{cases}\label{line:semijets2}
\end{equation}
where $\Phi(x,y,t):=\frac{|x-y|^4}{4\varepsilon^3\rho}$ and
\begin{align*}
q&:=D_x\Phi(\widehat{x}_{\varepsilon},\widetilde{x}_{\varepsilon},t_{\varepsilon})=-D_y\Phi(\widehat{x}_{\varepsilon},\widetilde{x}_{\varepsilon},t_{\varepsilon})=\delta(\widehat{x}_{\varepsilon}-\widetilde{x}_{\varepsilon})\quad\text{with}\quad\delta:=\frac{|\widehat{x}_{\varepsilon}-\widetilde{x}_{\varepsilon}|^2}{\varepsilon^3\rho},\\
A:&=D^2_{(x,y)}\Phi(\widehat{x}_{\varepsilon},\widetilde{x}_{\varepsilon},t_{\varepsilon})=\delta\begin{pmatrix}I_n+2\widehat{q}\otimes\widehat{q} &-I_n-2\widehat{q}\otimes\widehat{q} \\-I_n-2\widehat{q}\otimes\widehat{q} &I_n+2\widehat{q}\otimes\widehat{q} \end{pmatrix}\ \text{with}\  \widehat{q}:=\frac{q}{|q|}\ (\text{if }q\neq0).
\end{align*}

\medskip

By \cite[Lemma 13.1.(ii)]{CM14}, there exists $\overline{x}_{\varepsilon}\in\mathbb{R}^n$ such that
\begin{align}
(b_1,q,X)\in\overline{\mathcal{P}}^{2,+}\overline{u}^{\varepsilon}(\overline{x}_{\varepsilon},t_{\varepsilon}),\quad\widehat{u}^{\varepsilon}(\widehat{x}_{\varepsilon},t_{\varepsilon})=\overline{u}^{\varepsilon}(\overline{x}_{\varepsilon},t_{\varepsilon}),\quad|\widehat{x}_{\varepsilon}-\overline{x}_{\varepsilon}|\leq\varepsilon\eta,\label{line:overlinex2}
\end{align}
and consequently, by \eqref{line:uhat}, that
\begin{align}
\left(b_1-(-\lambda w_{\lambda}^{2\eta}(0)-\lambda_t)+\psi_t(\overline{P}_{\varepsilon}),q+D\psi(\overline{P}_{\varepsilon}),X+D^2\psi(\overline{P}_{\varepsilon})\right)\in\overline{\mathcal{P}}^{2,+}u^{\varepsilon}(\overline{x}_{\varepsilon},t_{\varepsilon})\label{line:}
\end{align}
with $\overline{P}_{\varepsilon}:=(\overline{x}_{\varepsilon},t_{\varepsilon})$.

Also, we note that $v_{\lambda}^{2\eta}(x):=\varepsilon\left(w_{\lambda}^{2\eta}\left(\frac{x}{\varepsilon}\right)-w_{\lambda}^{2\eta}(0)\right)$ is a viscosity solution to
\begin{align*}
F^{2\eta}\left(\varepsilon D^2v_{\lambda}^{2\eta}(x),p+Dv_{\lambda}^{2\eta}(x),\frac{x}{\varepsilon}\right)\geq\overline{\lambda}-\min\left\{\frac{1}{8}\theta,\frac{1}{8}\overline{\lambda}\right\}\quad\text{on}\quad\mathbb{R}^n,
\end{align*}
from \eqref{line:eqnofwlambdaeta2}, \eqref{line:convergenceofw2}. Therefore, from \eqref{line:ell} and the second line of \eqref{line:semijets2}, we have
\begin{align}
b_2=-\lambda w_{\lambda}^{2\eta}(0),\ (q-p,Y)\in\overline{\mathcal{J}}^{2,-}v_{\lambda}^{2\eta}(\widetilde{x}_{\varepsilon}).\label{line:semijetofv2}
\end{align}
Hence, by \eqref{line:} and \eqref{line:semijetofv2}, we obtain the viscosity inequalities
\begin{equation}
\begin{cases}
\lambda_t+\psi_t(\overline{P}_{\varepsilon})+F_*\left(\varepsilon X+\varepsilon D^2\psi(\overline{P}_{\varepsilon}),q+D\psi(\overline{P}_{\varepsilon}),\frac{\overline{x}_{\varepsilon}}{\varepsilon}\right)\leq0, \quad &\\
(F^{2\eta})^*\left(\varepsilon Y,q,\frac{\widetilde{x}_{\varepsilon}}{\varepsilon}\right)\geq\overline{\lambda}-\min\left\{\frac{1}{8}\theta,\frac{1}{8}\overline{\lambda}\right\}. \quad &
\end{cases}\label{line:viscosityinequalities2}
\end{equation}

\medskip

\textbf{Step 2.5: Obtain bounds of the gradient $q$ and of the Hessians $X,Y$.}

\medskip

Similarly as before, we set $\gamma=\frac{1}{3\delta}$ in \eqref{line:semijets2} and combine the fact that $\delta\leq\frac{\eta^2}{\varepsilon\rho}$ (from \eqref{line:overlinex2}) to obtain
\begin{align}
|q|\leq\frac{\eta^3}{\rho}\quad\text{and}\quad-\frac{9\eta^2}{\rho}I_n\leq \varepsilon X\leq \varepsilon Y\leq \frac{9\eta^2}{\rho}I_n.\label{line:boundfromabove2}
\end{align}

\medskip

We separate the gradient $q$ from the origin with a constant that depends only on $\theta,p$. Note that $(b_1,q,X)\in\overline{\mathcal{P}}^{2,+}\widehat{u}^{\varepsilon}(\widehat{x}_{\varepsilon},t_{\varepsilon})$ and that the function $\widehat{u}^{\varepsilon}$ is sup-convoluted by the definition \eqref{line:uhat}. Therefore, \cite[Lemma 13.1(ii)]{CM14} implies, together with \eqref{line:viscosityinequalities2}, \eqref{line:boundfromabove2}, that
\begin{align*}
\frac{1}{2}\overline{\lambda}\leq(F^{2\eta})^*\left(\varepsilon Y,q,\frac{\widetilde{x}_{\varepsilon}}{\varepsilon}\right)\leq(F^{2\eta})^*\left(\varepsilon X,q,\frac{\widetilde{x}_{\varepsilon}}{\varepsilon}\right)\leq c|q|
\end{align*}
for some constant $c=c(\eta)=c(\theta,p)>0$, which then yields
\begin{align}
\frac{\overline{\lambda}}{2c}\leq|q|.\label{line:separatefromtheorigin2}
\end{align}

\medskip

\textbf{Step 2.6: Derive a contradiction.}

\medskip

Note that the operator $F$ is uniformly continuous on
\begin{align*}
\left\{(X',p',y')\in S^n\times\mathbb{R}^n\times\mathbb{R}^n:\|X'\|\leq1+\frac{9\eta^2}{\rho},\ \frac{\overline{\lambda}}{4c}\leq|p'|\leq1+\frac{\eta^3}{\rho}\right\}.
\end{align*}
Combining the above uniform continuity with \eqref{line:widetildex}, \eqref{line:overlinex2}, we can choose $r=r(\theta,p)\in(0,1),$ $\varepsilon=\varepsilon(P_0,r,\theta,p)\in(0,1)$ satisfying
\begin{align}
\varepsilon\|D^2\psi(\overline{P}_{\varepsilon})\|,|D\psi(\overline{P}_{\varepsilon})|,|\psi_t(\overline{P}_{\varepsilon})|\leq\mu(r)\leq\min\left\{1,\frac{\theta}{4},\frac{\overline{\lambda}}{4c}\right\},\label{line:psi_t2}
\end{align}
so that, by \eqref{line:boundfromabove2}, \eqref{line:separatefromtheorigin2},
\begin{align*}
\|\varepsilon X\|,\|\varepsilon X+\varepsilon D^2\psi(\overline{P}_{\varepsilon})\|\leq1+\frac{9\eta^2}{\rho}\ \ \text{and}\ \ \frac{\overline{\lambda}}{4c}\leq|q|,|q+D\psi(\overline{P}_{\varepsilon})|\leq1+\frac{\eta^3}{\rho},
\end{align*}
and so that
\begin{align}
F\left(\varepsilon X,q,\frac{\overline{x}_{\varepsilon}}{\varepsilon}\right)\leq\frac{1}{4}\theta+F\left(\varepsilon X+\varepsilon D^2\psi(\overline{P}_{\varepsilon}),q+D\psi(\overline{P}_{\varepsilon}),\frac{\overline{x}_{\varepsilon}}{\varepsilon}\right).\label{line:uniformcontinuity2}
\end{align}
We derive a contradiction by linking the inequalities \eqref{line:overlinelambda2}, \eqref{line:viscosityinequalities2}, \eqref{line:boundfromabove2}, \eqref{line:psi_t2}, \eqref{line:uniformcontinuity2}:
\begin{align*}
\overline{\lambda}-\frac{1}{8}\theta&\leq\left(F^{2\eta}\right)^*\left(\varepsilon X,q,\frac{\widetilde{x}_{\varepsilon}}{\varepsilon}\right)\\
&= F^{2\eta}\left(\varepsilon X,q,\frac{\widetilde{x}_{\varepsilon}}{\varepsilon}\right)\\
&\leq F\left(\varepsilon X,q,\frac{\overline{x}_{\varepsilon}}{\varepsilon}\right)\\
&\leq\frac{1}{4}\theta+F\left(\varepsilon X+\varepsilon D^2\psi(\widetilde{P}_{\varepsilon}),q+D\psi(\widetilde{P}_{\varepsilon}),\frac{\overline{x}_{\varepsilon}}{\varepsilon}\right)\\
&\leq \frac{1}{4}\theta-\lambda_t-\psi_t(\overline{P}_{\varepsilon})\\
&\leq \frac{1}{4}\theta+\overline{\lambda}-\theta-\psi_t(\overline{P}_{\varepsilon})\\
&\leq \overline{\lambda}-\frac{1}{2}\theta,
\end{align*}
which completes the proof of Claim 2 in the case $\overline{\lambda}=\overline{F}(p)>0$.

\medskip

\end{proof}

We finish this section by proving the following proposition, which implies Corollary \ref{cor:homogenizationofgeqn} together with Theorem \ref{thm:homogenizationofgeometricequations}. The proof is a simple argument using Perron's method.

\begin{proposition}\label{prop:correctorofgeqn}
Let $d,A>0$ and $V(x)=A(-\cos x_2\sin x_1,\cos x_1\sin x_2)$ for $x=(x_1,x_2)\in\mathbb{R}^2$. Let $$F(X,p,y)=\left(|p|-d\text{tr}\left\{\left(I_2-\hat{p}\otimes\hat{p}\right)X\right\}\right)_++V(y)\cdot p$$ for $(X,p,y)\in S^2\times(\mathbb{R}^2\setminus\left\{0\right\})\times\mathbb{R}^2$. For each $p\in\mathbb{R}^2$, let $\overline{H}(p)$ be the unique real number as in the statement of Corollary \ref{cor:homogenizationofgeqn}. Then, \eqref{eq:cellproblem} with the right-hand side replaced by $\overline{H}(p)$ admits a $\mathbb{Z}^2$-periodic viscosity solution $v$.
\end{proposition}
\begin{proof}
We skip tracking the dependency on $d$ and $V$ as they are fixed. From \cite[Corollary 3.3]{GLXY22}, we see that there exist a viscosity subsolution $\overline{v}$ and a viscosity supersolution $\underline{v}$ to \eqref{eq:cellproblem} with the right-hand side replaced by $\overline{H}(p)$, and moreover, $$\sup_{\mathbb{R}^2}\left\{|\overline{v}|,|\underline{v}|\right\}\leq C_0$$ for some constant $C_0>0$ depending only on $p$, and the both are $\mathbb{Z}^2$-periodic. Then, $\overline{v}-C_0$ and $\underline{v}+C_0$ are also a viscosity sub and supersolution, respectively, and they satisfy $$-2C_0\leq \overline{v}-C_0\leq\underline{v}+C_0\leq 2C_0.$$
By Perron's method, namely by taking the supremum of $\mathbb{Z}^2$-periodic subsolutions between $\overline{v}-C_0$ and $\underline{v}+C_0$, we see that there exists a $\mathbb{Z}^2$-periodic solution $v$ to \eqref{eq:cellproblem} with the right-hand side replaced by $\overline{H}(p)$ satisfying
$$-2C_0\leq \overline{v}-C_0\leq v\leq\underline{v}+C_0\leq 2C_0.$$
\end{proof}

\section{Quantitative homogenization of the forced mean curvature equation}\label{sec:forcedmcf}

We turn our attention to the forced mean curvature equation in this section. We are interested in a rate of periodic homogenization of the flow. We provide the rate $O\left(\varepsilon^{1/8}\right)$ by proving Theorem \ref{thm:rateofforcedmcf} in this subsection under the coercivity condition \eqref{assumption:coercivity}.

\subsection{Proof of Theorem \ref{thm:rateofforcedmcf}}\label{subsec:provingrate}
Throughout this subsection, we let
\begin{align*}
c^{\eta}&:=\sup_{z\in\mathbb{R}^n:|z|\leq\eta}c(\cdot+z),\\ c_{\eta}&:=\inf_{z\in\mathbb{R}^n:|z|\leq\eta}c(\cdot+z),
\end{align*}
and
\begin{align*}
F(X,p,y)&:=-\mathrm{tr}\left\{\left(I_n-\widehat{p}\otimes\widehat{p}\right)X\right\}-c(y)|p|, \\
F^{\eta}(X,p,y)&:=-\mathrm{tr}\left\{\left(I_n-\widehat{p}\otimes\widehat{p}\right)X\right\}-c^{\eta}(y)|p|, \\
F_{\eta}(X,p,y)&:=-\mathrm{tr}\left\{\left(I_n-\widehat{p}\otimes\widehat{p}\right)X\right\}-c_{\eta}(y)|p|
\end{align*}
for $(X,p,y)\in S^n\times\left(\mathbb{R}^n\setminus\{0\}\right)\times\mathbb{R}^n$, $\eta\geq0$.

We follow the framework of \cite{C-DI01}. Before we get into the proof of Theorem \ref{thm:rateofforcedmcf}, we record Lipschitz estimates of solutions $u^{\varepsilon}$ and of perturbed approximate correctors $v^{\lambda,\eta},v^{\lambda}_{\eta}$ in the following proposition.

\begin{proposition}\label{prop:correctors}
Suppose that $c$ is $\mathbb{Z}^n$-periodic, Lipschitz continuous and satisfies \eqref{assumption:coercivity}.
\begin{itemize}
\item[(i)] Then, $c^{\eta},c_{\eta}$ are $\mathbb{Z}^n$-periodic, Lipschitz continuous and satisfiy \eqref{assumption:coercivity} (with the same $\delta>0$) as well for $\eta\geq0$.
\item[(ii)] Let $u_0$ be a function on $\mathbb{R}^n$ such that $\|Du_0\|_{C^{0,1}(\mathbb{R}^n)}<\infty$. For each $\varepsilon\in(0,1)$, there exists a unique viscosity solution $u^{\varepsilon}$ to \eqref{eq:uepsilon}, and $u^{\varepsilon}$ enjoys Lipschitz estimates:
\begin{align}
\|u^{\varepsilon}_t\|_{L^{\infty}(\mathbb{R}^n\times[0,\infty))}+\|Du^{\varepsilon}\|_{L^{\infty}(\mathbb{R}^n\times[0,\infty))}\leq M\label{line:Lipschitzestimates}
\end{align}
for some constant $M=M(n,\|c\|_{L^{\infty}(\mathbb{R}^n)},\|Du_0\|_{C^{0,1}(\mathbb{R}^n)},\delta)>0$.
\item[(iii)] Let $\lambda>0,\eta\geq0$. For each $p\in\mathbb{R}^n$, let $v^{\lambda,\eta}=v^{\lambda,\eta}(\cdot,p)$ denote the unique viscosity solution to
\begin{align*}
\lambda v^{\lambda,\eta}+F^{\eta}\left(D^2v^{\lambda,\eta},p+Dv^{\lambda,\eta},y\right)=0\quad\text{on }\mathbb{R}^n.
\end{align*}
Then, there exists a constant $C=C(n,\|c\|_{C^{0,1}(\mathbb{R}^n)},\delta)>0$ that for $x,y,p,q\in\mathbb{R}^n$,
\begin{equation}
\begin{cases}
|v^{\lambda,\eta}(x,p)-v^{\lambda,\eta}(y,p)|\leq C|p||x-y|, \quad & \\
|v^{\lambda,\eta}(x,p)-v^{\lambda,\eta}(x,q)|\leq \frac{C}{\lambda}|p-q|, \quad & \\
\|\lambda v^{\lambda,\eta}(\cdot,p)+\overline{F}(p)\|_{L^{\infty}(\mathbb{R}^n)}\leq C|p|(\lambda+\eta). \quad &
\end{cases}\label{line:correctors}
\end{equation}
Here, $\overline{F}(p)$ denotes the unique real number such that
\begin{align*}
F\left(D^2v,p+Dv,y\right)=\overline{F}(p)\quad\text{on }\mathbb{R}^n.
\end{align*}
admits a $\mathbb{Z}^n$-periodic viscosity solution. Also, the similar holds for a solution $v^{\lambda}_{\eta}=v^{\lambda}_{\eta}(\cdot,p)$ to
\begin{align*}
\lambda v^{\lambda}_{\eta}+F_{\eta}\left(D^2v^{\lambda}_{\eta},p+Dv^{\lambda}_{\eta},y\right)=0\quad\text{on }\mathbb{R}^n.
\end{align*}
\end{itemize}
\end{proposition}

We refer to \cite[Lemma 3.2]{LS05} for (iii). The statement (i) is an easy consequence of convolution, and (ii) can be shown by considering a vanishing viscosity parameter \cite[Theorem 1.13]{T21}. During the derivation of \eqref{line:Lipschitzestimates}, the time-derivative is bounded by Hessians, for which we refer to \cite[Appendix A.1]{DKY08}, \cite[Section 2]{J23}.

Now, we prove Theorem \ref{thm:rateofforcedmcf}.

\begin{proof}[Proof of Theorem \ref{thm:rateofforcedmcf}]
Throughout the proof, $C$ denotes positive constants varying line by line, and they depend only on $n,\|c\|_{C^{0,1}(\mathbb{R}^n)},\|Du_0\|_{C^{0,1}\left(\mathbb{R}^n\right)},\delta$, which we call the data from now on.

We first show that
\begin{align}
u^{\varepsilon}(x,t)-u(x,t)\leq C(1+T)\varepsilon^{1/8}\label{line:upperbound} 
\end{align}
for $(x,t)\in\mathbb{R}^n\times[0,T]$.

\medskip

\textbf{Step 0: The framework of doubling variable method with approximate correctors for quantification \cite{C-DI01}.}

\medskip

For a given $\varepsilon\in(0,1)$, we let $\eta=\varepsilon^{\beta}$ and let
\begin{align*}
\Phi(x,y,z,t,s):=u^{\varepsilon}(x,t)&-u(y,s)-\varepsilon v^{\lambda,\eta}\left(\frac{x}{\varepsilon},\frac{z-y}{\varepsilon^{\beta}}\right)\\
\qquad\qquad\qquad&-\frac{|x-y|^2+|t-s|^2}{2\varepsilon^{\beta}}-\frac{|x-z|^2}{2\varepsilon^{\beta}}-K(t+s)-\gamma_1\langle y\rangle,
\end{align*}
where $\lambda=\varepsilon^{\theta},\ K=K_1\varepsilon^{\beta},\ \gamma_1\in(0,\varepsilon^{\beta}]$, and $\theta,\beta\in(0,1], K_1>0$ are constants to be determined. Then, the global maximum of $\Phi$ on $\mathbb{R}^{3n}\times[0,T]^2$ is attained at a certain point, say at $(\hat{x},\hat{y},\hat{z},\hat{t},\hat{s})\in\mathbb{R}^{3n}\times[0,T]^2$ (abusing the notation $\hat{p}=\frac{p}{|p|}$ for $p\in\mathbb{R}^n$).

From $\Phi(\hat{x},\hat{y},\hat{z},\hat{t},\hat{s})\geq\Phi(\hat{x},\hat{y},\hat{x},\hat{t},\hat{s})$ with \eqref{line:correctors}, we have
$$
\frac{|\hat{x}-\hat{z}|^2}{2\varepsilon^{\beta}}\leq\varepsilon\left(v^{\lambda,\eta}\left(\frac{\hat{x}}{\varepsilon},\frac{\hat{x}-\hat{y}}{\varepsilon^{\beta}}\right)-v^{\lambda,\eta}\left(\frac{\hat{x}}{\varepsilon},\frac{\hat{z}-\hat{y}}{\varepsilon^{\beta}}\right)\right)\leq C\varepsilon^{1-\theta-\beta}|\hat{x}-\hat{z}|,
$$
which gives $|\hat{x}-\hat{z}|\leq C\varepsilon^{1-\theta}$. Also, from $\Phi(\hat{x},\hat{y},\hat{z},\hat{t},\hat{s})\geq\Phi(\hat{x},\hat{x},\hat{x},\hat{t},\hat{s})$, we get
\begin{align*}
\frac{|\hat{x}-\hat{y}|^2+|\hat{x}-\hat{z}|^2}{2\varepsilon^{\beta}}&\leq u(\hat{x},\hat{s})-u(\hat{y},\hat{s})+\varepsilon\left(v^{\lambda,\eta}\left(\frac{\hat{x}}{\varepsilon},0\right)-v^{\lambda,\eta}\left(\frac{\hat{x}}{\varepsilon},\frac{\hat{z}-\hat{y}}{\varepsilon^{\beta}}\right)\right)\\
&\leq C\left(|\hat{x}-\hat{y}|+|\hat{y}-\hat{z}|\right),
\end{align*}
as long as $\theta+\beta\leq1$. This yields $|\hat{x}-\hat{y}|+|\hat{y}-\hat{z}|\leq C\varepsilon^{\beta}$. Lastly, $\Phi(\hat{x},\hat{y},\hat{z},\hat{t},\hat{s})\geq\Phi(\hat{x},\hat{y},\hat{z},\hat{t},\hat{t})$ and $\Phi(\hat{x},\hat{y},\hat{z},\hat{t},\hat{s})\geq\Phi(\hat{x},\hat{y},\hat{z},\hat{s},\hat{s})$ give $|\hat{t}-\hat{s}|\leq C\varepsilon^{\beta}$. Thus, we have
\begin{equation}
\begin{cases}
|\hat{x}-\hat{z}|\leq C\varepsilon^{1-\theta}, \quad & \\
|\hat{x}-\hat{y}|+|\hat{y}-\hat{z}|+|\hat{t}-\hat{s}|\leq C\varepsilon^{\beta}. \quad & 
\end{cases}\label{line:hats}
\end{equation}

\medskip

\textbf{Claim. For $\beta=\frac{1}{8}$, $\theta\in\left[\frac{1}{8},\frac{3}{8}\right]$, there exists $K_1>0$ depending only on the data such that either $\hat{t}=0$ or $\hat{s}=0$.}

Once we establish this claim, we then obtain \eqref{line:upperbound}. Indeed, by \eqref{line:Lipschitzestimates}, \eqref{line:hats},
\begin{align*}
\Phi(x,x,x,t,t)\leq\Phi(\hat{x},\hat{y},\hat{z},\hat{t},\hat{s})\leq C\varepsilon^{1/8}.
\end{align*}
for all $(x,t)\in\mathbb{R}^n\times[0,T]$ in either case. Therefore, for all $(x,t)\in\mathbb{R}^n\times[0,T]$,
\begin{align*}
u^{\varepsilon}(x,t)-u(x,t)\leq\varepsilon v^{\lambda,\eta}\left(\frac{\hat{x}}{\varepsilon},0\right)+K_1\varepsilon^{1/8}t+\gamma_1\langle x\rangle+ C\varepsilon^{1/8}.
\end{align*}
Since it holds for arbitrary $\gamma_1\in(0,\varepsilon^{1/8}]$, we deduce \eqref{line:upperbound}.

\medskip

From now on, we suppose that $\hat{t},\hat{s}>0$. We postpone the explicit choice of $\beta,\theta$.

\medskip

By the fact that $(y,s)\mapsto\Phi(\hat{x},y,\hat{z},\hat{t},s)$ attains a maximum at $(\hat{y},\hat{s})$ and by the supersolution test of $u$ to \eqref{eq:ueffective}, for some
\begin{align}
-K+\frac{\hat{t}-\hat{s}}{\varepsilon^{\beta}}+\overline{F}\left(\frac{\hat{x}-\hat{y}}{\varepsilon^{\beta}}+\gamma_1\frac{\hat{y}}{\langle \hat{y}\rangle}-q\right)\geq0,\label{line:supersolutiontest}
\end{align}
for some $q\in\overline{\mathcal{D}}^{1,-}\left(\varepsilon v^{\lambda,\eta}\left(\frac{\hat{x}}{\varepsilon},\frac{\hat{z}-\cdot}{\varepsilon^{\beta}}\right)\right)$. See \cite[Lemma 2.4]{C-DI01} for the existence of $q$. Note that, by \eqref{line:correctors}, we have $|q|\leq C\varepsilon^{1-\theta-\beta}$.

\medskip

As the direction of the sup/inf-involution \cite[Lemma 13.1]{CM14} follows the sign of $\hat{t}-\hat{s}$, which shall be explained, we divide the cases accordingly.

\medskip

\textbf{Case 1. $\hat{t}\leq\hat{s}$.}

\medskip

\textbf{Step 1.1: Sup-involutions of auxiliary functions and their maximizers.}

\medskip

Let
\begin{align*}
\Psi(x,\xi,z,t)&:=\left(u^{\varepsilon}(x,t)-\frac{|x-\hat{y}|^2+|x-z|^2-2x\cdot(z-\hat{y})}{2\varepsilon^{\beta}}\right) \\
&\qquad\qquad-\varepsilon\left(v^{\lambda,\eta}\left(\xi,\frac{z-\hat{y}}{\varepsilon^{\beta}}\right)+\frac{z-\hat{y}}{\varepsilon^{\beta}}\cdot\xi\right)-(x-\varepsilon\xi)\cdot\frac{z-\hat{y}}{\varepsilon^{\beta}}\\
&\qquad\qquad\qquad\qquad-\frac{|x-\varepsilon\xi|^2}{2\alpha}-\frac{|z-\hat{z}|^2}{4\varepsilon^{\beta}}-\frac{|t-\hat{s}|^2+|t-\hat{t}|^2}{2\varepsilon^{\beta}}-Kt.
\end{align*}
Note that this auxiliary function is nothing but the terms of $\Phi$ involving $(x,z,t)$, keeping $(\hat{y},\hat{s})$ fixed, if we ignore for the term $-\frac{|z-\hat{z}|^2}{4\varepsilon^{\beta}}-\frac{|t-\hat{t}|^2}{2\varepsilon^{\beta}}$. This additional term is attached to quantify the distance between maximizers $(\overline{z}_{\alpha},\overline{t}_{\alpha})$, which will be taken soon, and $(\hat{z},\hat{t})$.

Let
\begin{align*}
\overline{\Psi}^{\mu}(x,\xi,z,t)&:=\sup_{w\in\overline{B}_{\varepsilon\mu}(x)}\left(u^{\varepsilon}(w,t)-\frac{|w-\hat{y}|^2+|w-z|^2-2w\cdot(z-\hat{y})}{2\varepsilon^{\beta}}\right) \\
&\qquad\qquad-\varepsilon\left(v^{\lambda,\eta}\left(\xi,\frac{z-\hat{y}}{\varepsilon^{\beta}}\right)+\frac{z-\hat{y}}{\varepsilon^{\beta}}\cdot\xi\right)-(x-\varepsilon\xi)\cdot\frac{z-\hat{y}}{\varepsilon^{\beta}}\\
&\qquad\qquad\qquad\qquad-\frac{|x-\varepsilon\xi|^2}{2\alpha}-\frac{|z-\hat{z}|^2}{4\varepsilon^{\beta}}-\frac{|t-\hat{s}|^2+|t-\hat{t}|^2}{2\varepsilon^{\beta}}-Kt,
\end{align*}
where $\mu=\frac{1}{2}\eta=\frac{1}{2}\varepsilon^{\beta}$, and $\alpha>0$ is to be determined later. Then, $\overline{\Psi}^{\mu}$ attains a global maximum, say at $(\overline{x}_{\alpha},\overline{\xi}_{\alpha},\overline{z}_{\alpha},\overline{t}_{\alpha})\in\mathbb{R}^{3n}\times[0,T]$.

\medskip

From $\overline{\Psi}^{\mu}(\overline{x}_{\alpha},\overline{\xi}_{\alpha},\overline{z}_{\alpha},\overline{t}_{\alpha})\geq\overline{\Psi}^{\mu}(\overline{x}_{\alpha},\frac{\overline{x}_{\alpha}}{\varepsilon},\overline{z}_{\alpha},\overline{t}_{\alpha})$ with \eqref{line:correctors}, \eqref{line:hats}, we have
\begin{align*}
\frac{|\overline{x}_{\alpha}-\varepsilon\overline{\xi}_{\alpha}|^2}{2\alpha}&\leq\varepsilon\left(v^{\lambda,\eta}\left(\frac{\overline{x}_{\alpha}}{\varepsilon},\frac{\overline{z}_{\alpha}-\hat{y}}{\varepsilon^{\beta}}\right)-v^{\lambda,\eta}\left(\overline{\xi}_{\alpha},\frac{\overline{z}_{\alpha}-\hat{y}}{\varepsilon^{\beta}}\right)\right)\\
&\leq C|\overline{x}_{\alpha}-\varepsilon\overline{\xi}_{\alpha}|\left|\frac{\overline{z}_{\alpha}-\hat{y}}{\varepsilon^{\beta}}\right|,
\end{align*}
and in turn,
\begin{align}
|\overline{x}_{\alpha}-\varepsilon\overline{\xi}_{\alpha}|\leq C\alpha\left|\frac{\overline{z}_{\alpha}-\hat{y}}{\varepsilon^{\beta}}\right|.\label{line:xi}
\end{align}

\medskip

Now, we estimate $|\overline{z}_{\alpha}-\hat{z}|$ and $|\overline{t}_{\alpha}-\hat{t}|$ by using the term $-\frac{|z-\hat{z}|^2}{4\varepsilon^{\beta}}-\frac{|t-\hat{t}|^2}{2\varepsilon^{\beta}}$.
\begin{align*}
&\overline{\Psi}^{\mu}(\overline{x}_{\alpha},\overline{\xi}_{\alpha},\overline{z}_{\alpha},\overline{t}_{\alpha})\\
&\leq C\varepsilon\mu+u^{\varepsilon}(\overline{x}_{\alpha},\overline{t}_{\alpha})+\sup_{w\in\overline{B}_{\varepsilon\mu}(\overline{x}_{\alpha})}\left(-\frac{|w-\hat{y}|^2+|w-\overline{z}_{\alpha}|^2-2(w-\overline{x}_{\alpha})\cdot(\overline{z}_{\alpha}-\hat{y})}{2\varepsilon^{\beta}}\right)\\
&\quad-\varepsilon v^{\varepsilon,\eta}\left(\overline{\xi}_{\alpha},\frac{\overline{z}_{\alpha}-\hat{y}}{\varepsilon^{\beta}}\right)-\frac{|\overline{x}_{\alpha}-\varepsilon\overline{\xi}_{\alpha}|^2}{2\alpha}-\frac{|\overline{z}_{\alpha}-\hat{z}|^2}{4\varepsilon^{\beta}}-\frac{|\overline{t}_{\alpha}-\hat{s}|^2+|\overline{t}_{\alpha}-\hat{t}|^2}{2\varepsilon^{\beta}}-K\overline{t}_{\alpha}\\
&\leq -\frac{|\overline{z}_{\alpha}-\hat{z}|^2+|\overline{t}_{\alpha}-\hat{t}|^2}{4\varepsilon^{\beta}}+C\varepsilon\mu+\varepsilon\left(v^{\lambda,\eta}\left(\frac{\overline{x}_{\alpha}}{\varepsilon},\frac{\overline{z}_{\alpha}-\hat{y}}{\varepsilon^{\beta}}\right)-v^{\lambda,\eta}\left(\overline{\xi}_{\alpha},\frac{\overline{z}_{\alpha}-\hat{y}}{\varepsilon^{\beta}}\right)\right)\\
&\quad+\sup_{w\in\overline{B}_{\varepsilon\mu}(\overline{x}_{\alpha})}\left(-\frac{2(w-\overline{x}_{\alpha})\cdot\left((w-\overline{x}_{\alpha})-2(\overline{z}_{\alpha}-\overline{x}_{\alpha})\right)}{2\varepsilon^{\beta}}\right)\\
&\ \ \quad+u^{\varepsilon}(\overline{x}_{\alpha},\overline{t}_{\alpha})-\frac{|\overline{x}_{\alpha}-\hat{y}|^2}{2\varepsilon^{\beta}}-\frac{|\overline{x}_{\alpha}-\overline{z}_{\alpha}|^2}{2\varepsilon^{\beta}}-\varepsilon v^{\varepsilon,\eta}\left(\frac{\overline{x}_{\alpha}}{\varepsilon},\frac{\overline{z}_{\alpha}-\hat{y}}{\varepsilon^{\beta}}\right)-\frac{|\overline{t}_{\alpha}-\hat{s}|^2}{2\varepsilon^{\beta}}-K\overline{t}_{\alpha}
\end{align*}
Here, we used the fact that $u^{\varepsilon}(w,\overline{t}_{\alpha})\leq u^{\varepsilon}(\overline{x}_{\alpha},\overline{t}_{\alpha})+C\varepsilon\mu$ for $w\in\overline{B}_{\varepsilon\mu}(\overline{x}_{\alpha})$ in the first inequality, and used the fact that
\begin{align*}
&(|w-\hat{y}|^2-|\overline{x}_{\alpha}-\hat{y}|^2)+(|w-\overline{z}_{\alpha}|^2-|\overline{x}_{\alpha}-\overline{z}_{\alpha}|^2)-2(w-\overline{x}_{\alpha})\cdot(\overline{z}_{\alpha}-\hat{y})\\
&=2(w-\overline{x}_{\alpha})\cdot\left((w-\overline{x}_{\alpha})-2(\overline{z}_{\alpha}-\overline{x}_{\alpha})\right),
\end{align*}
in the second inequality. The others are rearrangement of the terms. Now, from $\Phi(\hat{x},\hat{y},\hat{z},\hat{t},\hat{s})\geq\Phi(\overline{x}_{\alpha},\hat{y},\overline{z}_{\alpha},\overline{t}_{\alpha},\hat{s})$, we have
\begin{align*}
&\overline{\Psi}^{\mu}(\overline{x}_{\alpha},\overline{\xi}_{\alpha},\overline{z}_{\alpha},\overline{t}_{\alpha})\\
&\leq -\frac{|\overline{z}_{\alpha}-\hat{z}|^2+|\overline{t}_{\alpha}-\hat{t}|^2}{4\varepsilon^{\beta}}+C\varepsilon\mu+C|\overline{x}_{\alpha}-\varepsilon\overline{\xi}_{\alpha}|\left|\frac{\overline{z}_{\alpha}-\hat{y}}{\varepsilon^{\beta}}\right|+\varepsilon^{1-\beta}\mu(\varepsilon\mu+2\left|\overline{x}_{\alpha}-\overline{z}_{\alpha}\right|)\\
&\qquad\qquad+\underbrace{u^{\varepsilon}(\hat{x},\hat{t})-\varepsilon v^{\varepsilon,\eta}\left(\frac{\hat{x}}{\varepsilon},\frac{\hat{z}-\hat{y}}{\varepsilon^{\beta}}\right)-\frac{|\hat{x}-\hat{y}|^2}{2\varepsilon^{\beta}}-\frac{|\hat{x}-\hat{z}|^2}{2\varepsilon^{\beta}}-\frac{|\hat{t}-\hat{s}|^2}{2\varepsilon^{\beta}}-K\hat{t}}_{=\Psi\left(\hat{x},\frac{\hat{x}}{\varepsilon},\hat{z},\hat{t}\right)\leq\overline{\Psi}^{\mu}\left(\hat{x},\frac{\hat{x}}{\varepsilon},\hat{z},\hat{t}\right)\leq\overline{\Psi}^{\mu}(\overline{x}_{\alpha},\overline{\xi}_{\alpha},\overline{z}_{\alpha},\overline{t}_{\alpha})},
\end{align*}
and thus, by \eqref{line:xi} and by the inequality $\left|\frac{\overline{z}_{\alpha}-\hat{y}}{\varepsilon^{\beta}}\right|^2\leq2\left(C+\left|\frac{\overline{z}_{\alpha}-\hat{z}}{\varepsilon^{\beta}}\right|^2\right)$ from \eqref{line:hats},
\begin{align}
\frac{|\overline{z}_{\alpha}-\hat{z}|^2+|\overline{t}_{\alpha}-\hat{t}|^2}{4\varepsilon^{\beta}}&\leq C\varepsilon\mu+C\alpha+C\alpha\left|\frac{\overline{z}_{\alpha}-\hat{z}}{\varepsilon^{\beta}}\right|^2+\varepsilon^{2-\beta}\mu+2\varepsilon^{1-\beta}\mu\left|\overline{x}_{\alpha}-\overline{z}_{\alpha}\right|.\label{line:tildeandhat}
\end{align}

Choose $\overline{x}_{\alpha}^1\in\overline{B}_{\varepsilon\mu}(\overline{x}_{\alpha})$ such that the supremum
\begin{align*}
\sup_{w\in\overline{B}_{\varepsilon\mu}(\overline{x}_{\alpha})}\left(u^{\varepsilon}(w,\overline{t}_{\alpha})-\frac{|w-\hat{y}|^2+|w-\overline{z}_{\alpha}|^2-2(w-\overline{x}_{\alpha})\cdot(\overline{z}_{\alpha}-\hat{y})}{2\varepsilon^{\beta}}\right)
\end{align*}
is attained. From $\overline{\Psi}^{\mu}(\overline{x}_{\alpha},\overline{\xi}_{\alpha},\overline{z}_{\alpha},\overline{t}_{\alpha})\geq\overline{\Psi}^{\mu}(\overline{x}_{\alpha},\overline{\xi}_{\alpha},\overline{x}_{\alpha},\overline{t}_{\alpha})$, we get, by \eqref{line:correctors},
\begin{align*}
&2|\overline{x}_{\alpha}^1-\overline{z}_{\alpha}|^2-2|\overline{x}_{\alpha}^1-\overline{x}_{\alpha}|^2+|\overline{z}_{\alpha}-\hat{z}|^2-|\overline{x}_{\alpha}-\hat{z}|^2\\
&\leq4(\overline{x}_{\alpha}^1-\overline{x}_{\alpha})\cdot(\overline{z}_{\alpha}-\overline{x}_{\alpha})+4\varepsilon^{1+\beta}\left(v^{\lambda,\eta}\left(\overline{\xi}_{\alpha},\frac{\overline{x}_{\alpha}-\hat{y}}{\varepsilon^{\beta}}\right)-v^{\lambda,\eta}\left(\overline{\xi}_{\alpha},\frac{\overline{z}_{\alpha}-\hat{y}}{\varepsilon^{\beta}}\right)\right)\\
&\leq4(\overline{x}_{\alpha}^1-\overline{x}_{\alpha})\cdot(\overline{z}_{\alpha}-\overline{x}_{\alpha})+C\varepsilon^{1-\theta}|\overline{x}_{\alpha}-\overline{z}_{\alpha}|.
\end{align*}
By elementary calculations using
\begin{equation*}
\begin{cases}
|\overline{x}_{\alpha}^1-\overline{z}_{\alpha}|^2=|\overline{x}_{\alpha}^1-\overline{x}_{\alpha}|^2+2(\overline{x}_{\alpha}^1-\overline{x}_{\alpha})\cdot(\overline{x}_{\alpha}-\overline{z}_{\alpha})+|\overline{x}_{\alpha}-\overline{z}_{\alpha}|^2, \quad & \\
|\overline{x}_{\alpha}-\hat{z}|^2=|\overline{x}_{\alpha}-\overline{z}_{\alpha}|^2+2(\overline{x}_{\alpha}-\overline{z}_{\alpha})\cdot(\overline{z}_{\alpha}-\hat{z})+|\overline{z}_{\alpha}-\hat{z}|^2, \quad & 
\end{cases}
\end{equation*}
we see that there exists a constant $C_0>0$ depending only on the data such that \eqref{line:xi}, \eqref{line:tildeandhat} hold with $C_0$ in place of $C$,
\begin{align}
|\overline{x}_{\alpha}-\overline{z}_{\alpha}|\leq C_0\left(\varepsilon^{1-\theta}+|\overline{z}_{\alpha}-\hat{z}|\right).\label{line:xalphazalpha}
\end{align}

\medskip

Now, we consider $\alpha\in\left(0,\frac{1}{8C_0}\varepsilon^{1+\beta}\right)$ so that
\begin{align*}
\frac{|\overline{z}_{\alpha}-\hat{z}|^2}{8\varepsilon^{\beta}}&\leq C\left(\varepsilon^{1+\beta}+\varepsilon\left|\overline{x}_{\alpha}-\overline{z}_{\alpha}\right|\right)
\end{align*}
from \eqref{line:tildeandhat} with another constant $C>0$. Combining with \eqref{line:xalphazalpha}, we obtain
\begin{equation}
\begin{cases}
|\overline{z}_{\alpha}-\hat{z}|\leq C\varepsilon^{\frac{1}{2}+\beta}, \quad & \\
|\overline{x}_{\alpha}-\overline{z}_{\alpha}|\leq C\varepsilon^{\min\left\{\frac{1}{2}+\beta,1-\theta\right\}}, \quad & 
\end{cases}\label{line:zalphazhat}
\end{equation}
which in turn implies, again by \eqref{line:tildeandhat},
\begin{align}
|\overline{t}_{\alpha}-\hat{t}|\leq C\varepsilon^{\frac{1}{2}+\beta}.\label{line:talphathat}
\end{align}
Also, by \eqref{line:hats}, \eqref{line:tildeandhat} and \eqref{line:zalphazhat}, there exists a constant $C_1>4C_0$ depending only on the data such that
\begin{align*}
|\overline{x}_{\alpha}-\varepsilon\overline{\xi}_{\alpha}|\leq C_0\alpha\left|\frac{\overline{z}_{\alpha}-\hat{y}}{\varepsilon^{\beta}}\right|\leq C\alpha\left(1+\left|\frac{\overline{z}_{\alpha}-\hat{z}}{\varepsilon^{\beta}}\right|\right)\leq C_1\alpha.
\end{align*}
Now, we take $\alpha=\frac{1}{2C_1}\varepsilon^{1+\beta}=\frac{1}{2C_1}\varepsilon\eta\in\left(0,\frac{1}{8C_0}\varepsilon^{1+\beta}\right)$ so that
\begin{align}
\left|\frac{\overline{x}_{\alpha}}{\varepsilon}-\overline{\xi}_{\alpha}\right|\leq\frac{1}{2}\eta.\label{line:etabound1}
\end{align}

\medskip

\textbf{Step 1.2: The viscosity inequalities from the Crandall-Ishii lemma.}

\medskip

If $\overline{t}_{\alpha}=0$, we then necessarily have $\hat{t},\hat{s}\leq C\varepsilon^{\beta}$, which implies \eqref{line:upperbound} as before. We assume the other case $\overline{t}_{\alpha}>0$.

Let
\begin{equation*}
\begin{cases}
h(x,\xi,t):=\frac{|x-\varepsilon\xi|^2}{2\alpha}+(x-\varepsilon\xi)\cdot\frac{\overline{z}_{\alpha}-\hat{y}}{\varepsilon^{\beta}}+\frac{|t-\hat{s}|^2+|t-\hat{t}|^2}{2\varepsilon^{\beta}}+Kt+\frac{|\overline{z}_{\alpha}-\hat{z}|^2}{4\varepsilon^{\beta}}, \quad & \\
\overline{u}^{\varepsilon,\mu}(x,t):=\sup_{w\in\overline{B}_{\varepsilon\mu}(x)}\left(u^{\varepsilon}(w,t)-\frac{|w-\hat{y}|^2+|w-\overline{z}_{\alpha}|^2-2w\cdot(\overline{z}_{\alpha}-\hat{y})}{2\varepsilon^{\beta}}\right), \quad & \\
\widetilde{\ell}^{\varepsilon}(\xi):=v^{\lambda,\eta}\left(\xi,\frac{\overline{z}_{\alpha}-\hat{y}}{\varepsilon^{\beta}}\right)+\frac{\overline{z}_{\alpha}-\hat{y}}{\varepsilon^{\beta}}\cdot\xi \quad &
\end{cases}
\end{equation*}
so that
\begin{align*}
(x,\xi,t)\longmapsto\overline{\Psi}^{\mu}(x,\xi,\overline{z}_{\alpha},t)=\overline{u}^{\varepsilon,\mu}(x,t)-\varepsilon\widetilde{\ell}^{\varepsilon}(\xi)-h(x,\xi,t)
\end{align*}
attains a global maximum at $(\overline{x}_{\alpha},\overline{\xi}_{\alpha},\overline{t}_{\alpha})\in\mathbb{R}^{2n}\times(0,T]$.

\medskip

Since \cite[(8.5)]{CIL92} holds for our $F$ and for $\overline{u}^{\varepsilon,\mu}$, $\varepsilon\widetilde{\ell}^{\varepsilon}$, we can apply the Crandall-Ishii lemma \cite[Theorem 8.3]{CIL92} to see that for every $\gamma>0$, there exist $X,Y\in S^n$ such that
\begin{equation}
\begin{cases}
(b_1,p,X)\in\overline{\mathcal{P}}^{2,+}\overline{u}^{\varepsilon,\mu}(\overline{x}_{\alpha},\overline{t}_{\alpha}), \quad &\\
(b_2,q,Y)\in\overline{\mathcal{P}}^{2,-}\varepsilon\widetilde{\ell}^{\varepsilon}(\overline{\xi}_{\alpha}), \quad &\\
b_1=b_1-b_2=h_t(\overline{x}_{\alpha},\overline{\xi}_{\alpha},\overline{t}_{\alpha})=K+\frac{\overline{t}_{\alpha}-\hat{s}}{\varepsilon^{\beta}}+\frac{\overline{t}_{\alpha}-\hat{t}}{\varepsilon^{\beta}}, \quad &\\
-\left(\frac{1}{\gamma}+\|A\|\right)I_{2n}\leq\begin{pmatrix}X &0 \\0 &-Y \end{pmatrix}\leq A+\gamma A^2,
\end{cases}\label{line:semijetsforced}
\end{equation}
where
\begin{align*}
p&:=D_xh(\overline{x}_{\alpha},\overline{\xi}_{\alpha},\overline{t}_{\alpha})=\frac{\overline{x}_{\alpha}-\varepsilon\overline{\xi}_{\alpha}}{\alpha}+\frac{\overline{z}_{\alpha}-\hat{y}}{\varepsilon^{\beta}},\\
q&:=-D_{\xi}h(\overline{x}_{\alpha},\overline{\xi}_{\alpha},\overline{t}_{\alpha})=\varepsilon p,\\
A&:=D^2_{(x,\xi)}h(\overline{x}_{\alpha},\overline{\xi}_{\alpha},\overline{t}_{\alpha})=\frac{1}{\alpha}\begin{pmatrix}I_n & -\varepsilon I_n \\ -\varepsilon I_n & \varepsilon^2 I_n \end{pmatrix}.
\end{align*}
As $\|A\|$ is comparable to $\frac{1}{\alpha}$, we take $\gamma=\alpha$ so that we can deduce from \eqref{line:semijetsforced} that
\begin{equation}
\begin{cases}
-\frac{C}{\varepsilon^{\beta}}I_n\leq\varepsilon X\leq\frac{C}{\varepsilon^{\beta}}I_n, \quad & \\
\quad\hspace{2mm} \varepsilon X\leq\frac{1}{\varepsilon} Y \quad & 
\end{cases}\label{line:hessiansforced}
\end{equation}
with a constant $C>0$ depending only on the data.

\medskip

By the choice of $\overline{x}_{\alpha}^1\in\overline{B}_{\varepsilon\mu}(\overline{x}_{\alpha})$ and by the definition of $\overline{u}^{\varepsilon,\mu}$, we can apply \cite[Lemma 13.2]{CM14} to obtain
\begin{align*}
\left(b_1,p+\frac{2(\overline{x}_{\alpha}^1-\overline{z}_{\alpha})}{\varepsilon^{\beta}},X+\frac{2}{\varepsilon^{\beta}}I_n\right)\in\overline{\mathcal{P}}^{2,+}u^{\varepsilon}(\overline{x}_{\alpha}^1,\overline{t}_{\alpha}),
\end{align*}
which gives, from the subsolution test of $u^{\varepsilon}$ and \eqref{line:semijetsforced},
\begin{align}
K+\frac{\overline{t}_{\alpha}-\hat{s}}{\varepsilon^{\beta}}+\frac{\overline{t}_{\alpha}-\hat{t}}{\varepsilon^{\beta}}+F_*\left(\varepsilon X+2\varepsilon^{1-\beta}I_n,p+\frac{2(\overline{x}_{\alpha}^1-\overline{z}_{\alpha})}{\varepsilon^{\beta}},\frac{\overline{x}_{\alpha}^1}{\varepsilon}\right)\leq0.\label{line:subsolutiontest1}
\end{align}

Also, from the supersolution test of $\widetilde{\ell}^{\varepsilon}$ and \eqref{line:semijetsforced}, we have
\begin{align}
\lambda v^{\lambda,\eta}\left(\overline{\xi}_{\alpha},\frac{\overline{z}_{\alpha}-\hat{y}}{\varepsilon^{\beta}}\right)+\left(F^{\eta}\right)^*\left(\frac{1}{\varepsilon}Y,p,\overline{\xi}_{\alpha}\right)\geq0.\label{line:supersolutiontest1}
\end{align}

\medskip

\textbf{Step 1.3: Separation of the gradient $p$ from the origin.}

\medskip

By \eqref{line:correctors}, \eqref{line:supersolutiontest1} and by the fact that $\left|\frac{\overline{z}_{\alpha}-\hat{y}}{\varepsilon^{\beta}}\right|\leq C$, we have
\begin{align*}
C(\varepsilon^{\theta}+\varepsilon^{\beta})-\overline{F}\left(\frac{\overline{z}_{\alpha}-\hat{y}}{\varepsilon^{\beta}}\right)&\geq\lambda v^{\lambda,\eta}\left(\overline{\xi}_{\alpha},\frac{\overline{z}_{\alpha}-\hat{y}}{\varepsilon^{\beta}}\right)\\
&\geq\left(-F^{\eta}\right)_*\left(\frac{1}{\varepsilon}Y,p,\overline{\xi}_{\alpha}\right)\\
&\geq\left(-F^{\eta}\right)_*\left(\varepsilon X,p,\overline{\xi}_{\alpha}\right)\\
&\geq-\frac{C}{\mu}|p|,
\end{align*}
where we used \cite[Lemma 13.1]{CM14} for $(b_1,p,X)\in\overline{\mathcal{P}}^{2,+}\overline{u}^{\varepsilon,\mu}(\overline{x}_{\alpha},\overline{t}_{\alpha})$ in the last inequality. On the other hand, by \eqref{line:correctors}, \eqref{line:supersolutiontest}, \eqref{line:hats}, \eqref{line:zalphazhat}, it holds that
\begin{align*}
-K&\geq-K+\frac{\hat{t}-\hat{s}}{\varepsilon^{\beta}}\\
&\geq-\overline{F}\left(\frac{\hat{x}-\hat{y}}{\varepsilon^{\beta}}+\gamma_1\frac{\hat{y}}{\langle\hat{y}\rangle}-q\right)\\
&\geq-\overline{F}\left(\frac{\overline{z}_{\alpha}-\hat{y}}{\varepsilon^{\beta}}\right)-C\varepsilon^{\min\left\{\frac{1}{2},\beta,1-\theta-\beta\right\}}.
\end{align*}
Linking the two inequalities, we see that there exists a constant $C>0$ depending only on the data such that
\begin{align*}
|p|\geq \varepsilon^{\beta}\left(C^{-1}K_1\varepsilon^{\beta}-C\varepsilon^{\min\left\{\frac{1}{2},\theta,\beta,1-\theta-\beta\right\}}\right).
\end{align*}
We require that $K_1>C^2$ and $\beta\leq\min\left\{\frac{1}{2},\theta,1-\theta-\beta\right\}$ so that
\begin{align}
|p|\geq\left(C^{-1}K_1-C\right)\varepsilon^{2\beta}.\label{line:separationfromtheoriginforced}
\end{align}
\medskip

\textbf{Step 1.4: Deriving a contradiction for a large constant $K_1>0$.}

\medskip

Note that $\left|\frac{\overline{x}_{\alpha}^1}{\varepsilon}-\overline{\xi}_{\alpha}\right|\leq\left|\frac{\overline{x}_{\alpha}^1}{\varepsilon}-\frac{\overline{x}_{\alpha}}{\varepsilon}\right|+\left|\frac{\overline{x}_{\alpha}}{\varepsilon}-\overline{\xi}_{\alpha}\right|\leq\eta$ by \eqref{line:etabound1} and the fact that $\overline{x}_{\alpha}^1\in\overline{B}_{\varepsilon\mu}(\overline{x}_{\alpha})$. Therefore, by connecting the viscosity inequalities \eqref{line:subsolutiontest1}, \eqref{line:supersolutiontest1}, we obtain, by \eqref{line:correctors}, that
\begin{align}
C(\varepsilon^{\theta}+\varepsilon^{\beta})-\overline{F}\left(\frac{\overline{z}_{\alpha}-\hat{y}}{\varepsilon^{\beta}}\right)&\geq\lambda v^{\lambda,\eta}\left(\overline{\xi}_{\alpha},\frac{\overline{z}_{\alpha}-\hat{y}}{\varepsilon^{\beta}}\right)\notag\\
&\geq-F^{\eta}\left(\frac{1}{\varepsilon}Y,p,\overline{\xi}_{\alpha}\right)\notag\\
&\geq-F\left(\varepsilon X,p,\frac{\overline{x}_{\alpha}^1}{\varepsilon}\right)\notag\\
&\geq-F\left(\varepsilon X+2\varepsilon^{1-\beta}I_n,p+\frac{2(\overline{x}_{\alpha}^1-\overline{z}_{\alpha})}{\varepsilon^{\beta}},\frac{\overline{x}_{\alpha}^1}{\varepsilon}\right)+E_1+E_2\notag\\
&\geq K+\frac{\overline{t}_{\alpha}^1-\hat{s}}{\varepsilon^{\beta}}+\frac{\overline{t}_{\alpha}^1-\hat{t}}{\varepsilon^{\beta}}+E_1+E_2\label{line:chainofinequalities}
\end{align}
where
\begin{align*}
E_1&:=F\left(\varepsilon X,p+\frac{2(\overline{x}_{\alpha}^1-\overline{z}_{\alpha})}{\varepsilon^{\beta}},\frac{\overline{x}_{\alpha}^1}{\varepsilon}\right)-F\left(\varepsilon X,p,\frac{\overline{x}_{\alpha}^1}{\varepsilon}\right),\\
E_2&:=F\left(\varepsilon X+2\varepsilon^{1-\beta}I_n,p+\frac{2(\overline{x}_{\alpha}^1-\overline{z}_{\alpha})}{\varepsilon^{\beta}},\frac{\overline{x}_{\alpha}^1}{\varepsilon}\right)-F\left(\varepsilon X,p+\frac{2(\overline{x}_{\alpha}^1-\overline{z}_{\alpha})}{\varepsilon^{\beta}},\frac{\overline{x}_{\alpha}^1}{\varepsilon}\right).
\end{align*}
Note that, from $|\overline{x}_{\alpha}^1-\overline{z}_{\alpha}|\leq|\overline{x}_{\alpha}^1-\overline{x}_{\alpha}|+|\overline{x}_{\alpha}-\overline{z}_{\alpha}|\leq C\varepsilon^{\min\left\{\frac{1}{2}+\beta,1-\theta\right\}}$ by \eqref{line:zalphazhat}. Therefore, we have, by \eqref{line:separationfromtheoriginforced},
\begin{align*}
\left|p+\frac{2(\overline{x}_{\alpha}^1-\overline{z}_{\alpha})}{\varepsilon^{\beta}}\cdot\nu\right|\geq\left(C^{-1}K_1-C\right)\varepsilon^{2\beta}
\end{align*}
for $\nu\in[0,1]$ if we require $K_1>C^2$, $2\beta\leq\min\left\{\frac{1}{2},1-\theta-\beta\right\}$ (and also $\beta\leq\theta$ from the previous requirement). This implies, with \eqref{line:hessiansforced}, that
\begin{align*}
|E_1|\leq C\varepsilon^{-\beta}\left((C^{-1}K_1-C)\varepsilon^{2\beta}\right)^{-1}\varepsilon^{\min\left\{\frac{1}{2},1-\theta-\beta\right\}},
\end{align*}
and therefore, we see that there exists a constant $C>0$ depending only on the data such that if $K_1>C$, then
\begin{equation}
\begin{cases}
|E_1|\leq\frac{C}{K_1-C}\varepsilon^{\min\left\{\frac{1}{2},1-\theta-\beta\right\}-3\beta}, \quad & \\
|E_2|\leq 2n\varepsilon^{1-\beta}. \quad & 
\end{cases}\label{line:commuatatorestimateforced}
\end{equation}

Therefore, by \eqref{line:supersolutiontest}, \eqref{line:chainofinequalities}, \eqref{line:commuatatorestimateforced}, we have
\begin{align*}
2K_1\varepsilon^{\beta}\leq C\left(\varepsilon^{\min\left\{\theta,\beta\right\}}+\frac{1}{K_1-C}\varepsilon^{\min\left\{\frac{1}{2},1-\theta-\beta\right\}-3\beta}\right)
\end{align*}
for some constant $C>0$ (with a larger one if necessary) depending only on the data. Now, we take $\beta=\frac{1}{8}$ and any $\theta\in\left[\frac{1}{8},\frac{3}{8}\right]$ as an optimal choice. Then, taking $K_1=C+1$ yields a contradiction. Therefore, there exists a constant $K_1>0$ depending only on the data such that if $\hat{t}\leq\hat{s}$, then $\hat{t}=0$.

\medskip

\textbf{Case 2. $\hat{t}\geq\hat{s}$.}

\medskip

\textbf{Step 2.1: Inf-involutions of auxiliary functions and their maximizers.}

\medskip

Let
\begin{align*}
\Psi_1(x,\xi,z,t):=u^{\varepsilon}(x,t)&-\left(\varepsilon v^{\lambda,\eta}\left(\xi,\frac{z-\hat{y}}{\varepsilon^{\beta}}\right)+\frac{|\varepsilon \xi-\hat{y}|^2+|\varepsilon \xi-z|^2}{2\varepsilon^{\beta}}\right)\\
&\qquad-\frac{|x-\varepsilon\xi|^2}{2\alpha}-\frac{|z-\hat{z}|^2}{4\varepsilon^{\beta}}-\frac{|t-\hat{s}|^2+|t-\hat{t}|^2}{2\varepsilon^{\beta}}-Kt.
\end{align*}
Also, we let
\begin{align*}
\overline{\Psi}_1^{\mu}(x,\xi,z,t):=u^{\varepsilon}(x,t)&-\inf_{w\in\overline{B}_{\mu}(\xi)}\left(\varepsilon v^{\lambda,\eta}\left(w,\frac{z-\hat{y}}{\varepsilon^{\beta}}\right)+\frac{|\varepsilon w-\hat{y}|^2+|\varepsilon w-z|^2}{2\varepsilon^{\beta}}\right)\\
&\qquad-\frac{|x-\varepsilon\xi|^2}{2\alpha}-\frac{|z-\hat{z}|^2}{4\varepsilon^{\beta}}-\frac{|t-\hat{s}|^2+|t-\hat{t}|^2}{2\varepsilon^{\beta}}-Kt,
\end{align*}
where $\mu=\frac{1}{2}\eta=\frac{1}{2}\varepsilon^{\beta}$, and $\alpha>0$ is to be determined later. Then, $\overline{\Psi}_1^{\mu}$ attains a global maximum, say at $(\overline{x}_{\alpha},\overline{\xi}_{\alpha},\overline{z}_{\alpha},\overline{t}_{\alpha})\in\mathbb{R}^{3n}\times[0,T]$ (by abuse of notations).

\medskip

From $\overline{\Psi}_1^{\mu}(\overline{x}_{\alpha},\overline{\xi}_{\alpha},\overline{z}_{\alpha},\overline{t}_{\alpha})\geq\overline{\Psi}_1^{\mu}(\varepsilon\overline{\xi}_{\alpha},\overline{\xi}_{\alpha},\overline{z}_{\alpha},\overline{t}_{\alpha})$, we have
\begin{align*}
u^{\varepsilon}(\overline{x}_{\alpha},\overline{t}_{\alpha})-\frac{|\overline{x}_{\alpha}-\varepsilon\overline{\xi}_{\alpha}|^2}{2\alpha}\geq u^{\varepsilon}(\varepsilon\overline{\xi}_{\alpha},\overline{t}_{\alpha}),
\end{align*}
which implies
\begin{align}
\left|\overline{\xi}_{\alpha}-\frac{\overline{x}_{\alpha}}{\varepsilon}\right|\leq\frac{1}{2}\eta\label{line:etabound2}
\end{align}
with $\alpha:=\frac{1}{2(C+1)}\varepsilon\eta$. Here, $C>0$ depends only on the Lipschitz constant of $u^{\varepsilon}$.

We estimate $|\overline{z}_{\alpha}-\hat{z}|$ and $|\overline{t}_{\alpha}-\hat{t}|$ by using the term $-\frac{|z-\hat{z}|^2}{4\varepsilon^{\beta}}-\frac{|t-\hat{t}|^2}{2\varepsilon^{\beta}}$ as in Case 1. First of all, it holds that
\begin{align*}
&\overline{\Psi}_1^{\mu}(\overline{x}_{\alpha},\overline{\xi}_{\alpha},\overline{z}_{\alpha},\overline{t}_{\alpha})\\
&\leq-\frac{|\overline{z}_{\alpha}-\hat{z}|^2+|\overline{t}_{\alpha}-\hat{t}|^2}{4\varepsilon^{\beta}}+\left(u^{\varepsilon}(\overline{x}_{\alpha},\overline{t}_{\alpha})-u^{\varepsilon}(\varepsilon\overline{\xi}_{\alpha},\overline{t}_{\alpha})\right)\\
&\quad\quad-\inf_{w\in\overline{B}_{\mu}(\overline{\xi}_{\alpha})}\left(\varepsilon v^{\lambda,\eta}\left(w,\frac{\overline{z}_{\alpha}-\hat{y}}{\varepsilon^{\beta}}\right)-\varepsilon v^{\lambda,\eta}\left(\overline{\xi}_{\alpha},\frac{\overline{z}_{\alpha}-\hat{y}}{\varepsilon^{\beta}}\right)\right)\\
&\qquad\qquad -\inf_{w\in\overline{B}_{\mu}(\overline{\xi}_{\alpha})}\left(\frac{\left(|\varepsilon w-\hat{y}|^2-|\varepsilon\overline{\xi}_{\alpha}-\hat{y}|^2\right)-\left(|\varepsilon w-\overline{z}_{\alpha}|^2-|\varepsilon\overline{\xi}_{\alpha}-\overline{z}_{\alpha}|^2\right)}{2\varepsilon^{\beta}}\right)\\
&+u^{\varepsilon}(\varepsilon\overline{\xi}_{\alpha},\overline{t}_{\alpha})-\varepsilon v^{\lambda,\eta}\left(\overline{\xi}_{\alpha},\frac{\overline{z}_{\alpha}-\hat{y}}{\varepsilon^{\beta}}\right)-\frac{|\varepsilon\overline{\xi}_{\alpha}-\hat{y}|^2+|\varepsilon\overline{\xi}_{\alpha}-\overline{z}_{\alpha}|^2}{2\varepsilon^{\beta}}-\frac{|\overline{t}_{\alpha}-\hat{s}|^2}{2\varepsilon^{\beta}}-K\overline{t}_{\alpha}.
\end{align*}
We note that $\Phi(\hat{x},\hat{y},\hat{z},\hat{t},\hat{s})\geq\Phi(\varepsilon\overline{\xi}_{\alpha},\hat{y},\overline{z}_{\alpha},\overline{t}_{\alpha},\hat{s})$ and
\begin{align*}
&\left(|\varepsilon w-\hat{y}|^2-|\varepsilon\overline{\xi}_{\alpha}-\hat{y}|^2\right)-\left(|\varepsilon w-\overline{z}_{\alpha}|^2-|\varepsilon\overline{\xi}_{\alpha}-\overline{z}_{\alpha}|^2\right)\\
&=2\varepsilon(w-\overline{\xi}_{\alpha})\cdot\left(\varepsilon(w-\overline{\xi}_{\alpha})+(\varepsilon\overline{\xi}_{\alpha}-\hat{y})+(\varepsilon\overline{\xi}_{\alpha}-\overline{z}_{\alpha})\right).
\end{align*}
By these facts, together with \eqref{line:Lipschitzestimates}, \eqref{line:correctors}, we have
\begin{align*}
&\overline{\Psi}_1^{\mu}(\overline{x}_{\alpha},\overline{\xi}_{\alpha},\overline{z}_{\alpha},\overline{t}_{\alpha})+\frac{|\overline{z}_{\alpha}-\hat{z}|^2+|\overline{t}_{\alpha}-\hat{t}|^2}{4\varepsilon^{\beta}}\\
&\leq C\varepsilon^{1+\beta}+C\varepsilon^{1+\beta}\left|\frac{\overline{z}_{\alpha}-\hat{y}}{\varepsilon^{\beta}}\right|+\varepsilon\left(\varepsilon^{1+\beta}+|\varepsilon\overline{\xi}_{\alpha}-\hat{y}|+|\varepsilon\overline{\xi}_{\alpha}-\overline{z}_{\alpha}|\right)\\
&+\underbrace{u^{\varepsilon}(\hat{x},\hat{t})-\varepsilon v^{\lambda,\eta}\left(\frac{\hat{x}}{\varepsilon},\frac{\hat{z}-\hat{y}}{\varepsilon^{\beta}}\right)-\frac{|\hat{x}-\hat{y}|^2+|\hat{x}-\hat{z}|^2}{2\varepsilon^{\beta}}-\frac{|\hat{t}-\hat{s}|^2}{2\varepsilon^{\beta}}-K\hat{t}}_{=\Psi_1\left(\hat{x},\frac{\hat{x}}{\varepsilon},\hat{z},\hat{t}\right)\leq\overline{\Psi}_1^{\mu}\left(\hat{x},\frac{\hat{x}}{\varepsilon},\hat{z},\hat{t}\right)\leq\overline{\Psi}_1^{\mu}(\overline{x}_{\alpha},\overline{\xi}_{\alpha},\overline{z}_{\alpha},\overline{t}_{\alpha})},
\end{align*}
which then yields, with \eqref{line:hats},
\begin{align}
\frac{|\overline{z}_{\alpha}-\hat{z}|^2+|\overline{t}_{\alpha}-\hat{t}|^2}{4\varepsilon^{\beta}}\leq C(\varepsilon^{1+\beta}+\varepsilon|\overline{z}_{\alpha}-\hat{z}|+\varepsilon|\varepsilon\overline{\xi}_{\alpha}-\overline{z}_{\alpha}|).\label{line:tildeandhat2}
\end{align}

\medskip

Choose $\overline{\xi}_{\alpha}^1\in\overline{B}_{\mu}(\overline{\xi}_{\alpha})$ such that the infimum
\begin{align*}
\inf_{w\in\overline{B}_{\mu}(\overline{\xi}_{\alpha})}\left(\varepsilon v^{\lambda,\eta}\left(w,\frac{\overline{z}_{\alpha}-\hat{y}}{\varepsilon^{\beta}}\right)+\frac{|\varepsilon w-\hat{y}|^2+|\varepsilon w-\overline{z}_{\alpha}|^2}{2\varepsilon^{\beta}}\right)
\end{align*}
is attained. Then, $\overline{\Psi}_{\mu}^1(\overline{x}_{\alpha},\overline{\xi}_{\alpha},\overline{z}_{\alpha},\overline{t}_{\alpha})\geq\overline{\Psi}_{\mu}^1(\overline{x}_{\alpha},\overline{\xi}_{\alpha},\varepsilon\overline{\xi}_{\alpha},\overline{t}_{\alpha})$ yields, with \eqref{line:correctors},
\begin{align*}
2|\varepsilon\overline{\xi}_{\alpha}^1-\overline{z}_{\alpha}|^2-2|\varepsilon\overline{\xi}_{\alpha}^1-\varepsilon\overline{\xi}_{\alpha}|^2+|\overline{z}_{\alpha}-\hat{z}|^2-|\varepsilon\overline{\xi}_{\alpha}-\hat{z}|^2\leq C\varepsilon^{1-\theta}|\varepsilon\overline{\xi}_{\alpha}-\overline{z}_{\alpha}|.
\end{align*}
By elementary calculations using
\begin{equation*}
\begin{cases}
|\varepsilon\overline{\xi}_{\alpha}^1-\overline{z}_{\alpha}|^2=|\varepsilon\overline{\xi}_{\alpha}^1-\varepsilon\overline{\xi}_{\alpha}|^2+2(\varepsilon\overline{\xi}_{\alpha}^1-\varepsilon\overline{\xi}_{\alpha})\cdot(\varepsilon\overline{\xi}_{\alpha}-\overline{z}_{\alpha})+|\varepsilon\overline{\xi}_{\alpha}-\overline{z}_{\alpha}|^2, \quad & \\
|\varepsilon\overline{\xi}_{\alpha}-\hat{z}|^2=|\varepsilon\overline{\xi}_{\alpha}-\overline{z}_{\alpha}|^2+2(\varepsilon\overline{\xi}_{\alpha}-\overline{z}_{\alpha})\cdot(\overline{z}_{\alpha}-\hat{z})+|\overline{z}_{\alpha}-\hat{z}|^2, \quad & 
\end{cases}
\end{equation*}
we see that
\begin{align*}
|\varepsilon\overline{\xi}_{\alpha}-\overline{z}_{\alpha}|\leq C\left(\varepsilon^{1-\theta}+|\overline{z}_{\alpha}-\hat{z}|\right).
\end{align*}
Combining this with \eqref{line:tildeandhat2}, we obtain
\begin{equation}
\begin{cases}
|\overline{z}_{\alpha}-\hat{z}|\leq C\varepsilon^{\frac{1}{2}+\beta}, \quad & \\
|\varepsilon\overline{\xi}_{\alpha}-\overline{z}_{\alpha}|\leq C\varepsilon^{\min\left\{\frac{1}{2}+\beta,1-\theta\right\}}, \quad & 
\end{cases}\label{line:zalphazhat2}
\end{equation}
which in turn implies, again by \eqref{line:tildeandhat2},
\begin{align}
|\overline{t}_{\alpha}-\hat{t}|\leq C\varepsilon^{\frac{1}{2}+\beta}.\label{line:talphathat2}
\end{align}

\medskip

\textbf{Step 2.2: The viscosity inequalities from the Crandall-Ishii lemma.}

\medskip

If $\overline{t}_{\alpha}=0$, we then necessarily have $\hat{t},\hat{s}\leq C\varepsilon^{\beta}$, which implies \eqref{line:upperbound} as before. We assume the other case $\overline{t}_{\alpha}>0$.

Let
\begin{equation*}
\begin{cases}
h(x,\xi,t):=\frac{|x-\varepsilon\xi|^2}{2\alpha}+\frac{|t-\hat{s}|^2+|t-\hat{t}|^2}{2\varepsilon^{\beta}}+Kt+\frac{|\overline{z}_{\alpha}-\hat{z}|^2}{4\varepsilon^{\beta}}, \quad & \\
\overline{\ell}^{\varepsilon,\mu}(\xi):=\inf_{w\in\overline{B}_{\mu}(\xi)}\left(\varepsilon v^{\lambda,\eta}\left(w,\frac{\overline{z}_{\alpha}-\hat{y}}{\varepsilon^{\beta}}\right)+\frac{|\varepsilon w-\hat{y}|^2+|\varepsilon w-\overline{z}_{\alpha}|^2}{2\varepsilon^{\beta}}\right), \quad &
\end{cases}
\end{equation*}
so that
\begin{align*}
(x,\xi,t)\longmapsto\overline{\Psi}^{\mu}_1(x,\xi,\overline{z}_{\alpha},t)=u^{\varepsilon}(x,t)-\overline{\ell}^{\varepsilon,\mu}(\xi)-h(x,\xi,t)
\end{align*}
attains a global maximum at $(\overline{x}_{\alpha},\overline{\xi}_{\alpha},\overline{t}_{\alpha})\in\mathbb{R}^{2n}\times(0,T]$.
\medskip

Since \cite[(8.5)]{CIL92} holds for our $F$ and for $u^{\varepsilon}$, $\overline{\ell}^{\varepsilon,\mu}$, we can apply the Crandall-Ishii lemma \cite[Theorem 8.3]{CIL92} to see that for every $\gamma>0$, there exist $X,Y\in S^n$ such that
\begin{equation}
\begin{cases}
(b_1,p,X)\in\overline{\mathcal{P}}^{2,+}u^{\varepsilon}(\overline{x}_{\alpha},\overline{t}_{\alpha}), \quad &\\
(b_2,q,Y)\in\overline{\mathcal{P}}^{2,-}\overline{\ell}^{\varepsilon,\mu}(\overline{\xi}_{\alpha}), \quad &\\
b_1=b_1-b_2=h_t(\overline{x}_{\alpha},\overline{\xi}_{\alpha},\overline{t}_{\alpha})=K+\frac{\overline{t}_{\alpha}-\hat{s}}{\varepsilon^{\beta}}+\frac{\overline{t}_{\alpha}-\hat{t}}{\varepsilon^{\beta}}, \quad &\\
-\left(\frac{1}{\gamma}+\|A\|\right)I_{2n}\leq\begin{pmatrix}X &0 \\0 &-Y \end{pmatrix}\leq A+\gamma A^2,
\end{cases}\label{line:semijetsforced2}
\end{equation}
where
\begin{align*}
p&:=D_xh(\overline{x}_{\alpha},\overline{\xi}_{\alpha},\overline{t}_{\alpha})=\frac{\overline{x}_{\alpha}-\varepsilon\overline{\xi}_{\alpha}}{\alpha},\\
q&:=-D_{\xi}h(\overline{x}_{\alpha},\overline{\xi}_{\alpha},\overline{t}_{\alpha})=\varepsilon p,\\
A&:=D^2_{(x,\xi)}h(\overline{x}_{\alpha},\overline{\xi}_{\alpha},\overline{t}_{\alpha})=\frac{1}{\alpha}\begin{pmatrix}I_n & -\varepsilon I_n \\ -\varepsilon I_n & \varepsilon^2 I_n \end{pmatrix}.
\end{align*}
Taking $\gamma=\alpha$, we have, from \eqref{line:semijetsforced2}, that
\begin{equation}
\begin{cases}
-\frac{C}{\varepsilon^{\beta}}I_n\leq\varepsilon X\leq\frac{C}{\varepsilon^{\beta}}I_n, \quad & \\
\quad\hspace{2mm} \varepsilon X\leq\frac{1}{\varepsilon} Y \quad & 
\end{cases}\label{line:hessiansforced2}
\end{equation}
as before.

\medskip

From the viscosity subsolution test to $u^{\varepsilon}$ at $(\overline{x}_{\alpha},\overline{t}_{\alpha})$,
\begin{align}
K+\frac{\overline{t}_{\alpha}-\hat{t}}{\varepsilon^{\beta}}+\frac{\overline{t}_{\alpha}-\hat{s}}{\varepsilon^{\beta}}+F_*\left(\varepsilon X,p,\frac{\overline{x}_{\alpha}}{\varepsilon}\right)\leq0.\label{line:subsolutiontest2}
\end{align}
Also, by the choice of $\overline{\xi}_{\alpha}^1\in\overline{B}_{\mu}(\overline{\xi}_{\alpha})$ and by the definition of $\overline{\ell}^{\varepsilon,\mu}$, we can apply \cite[Lemma 13.2]{CM14} to obtain
\begin{align*}
\left(b_2,q-\varepsilon^{1-\beta}\left((\varepsilon\overline{\xi}_{\alpha}^1-\hat{y})+(\varepsilon\overline{\xi}_{\alpha}^1-\overline{z}_{\alpha})\right),Y-2\varepsilon^{2-\beta}I_n\right)\in\overline{\mathcal{P}}^{2,-}\varepsilon v^{\lambda,\eta}\left(\overline{\xi}_{\alpha}^1,\frac{\overline{z}_{\alpha}-\hat{y}}{\varepsilon^{\beta}}\right),
\end{align*}
which gives, from the supersolution test to $v^{\lambda,\eta}$,
\begin{align}
\lambda v^{\lambda,\eta}\left(\overline{\xi}_{\alpha}^1,\frac{\overline{z}_{\alpha}-\hat{y}}{\varepsilon^{\beta}}\right)+\left(F^{\eta}\right)^*\left(\frac{1}{\varepsilon}Y-2\varepsilon^{1-\beta}I_n,p-\frac{2(\varepsilon\overline{\xi}_{\alpha}^1-\overline{z}_{\alpha})}{\varepsilon^{\beta}},\frac{\overline{x}_{\alpha}^1}{\varepsilon}\right)\geq0.\label{line:supersolutiontest2}
\end{align}

\textbf{Step 2.3: Separation of the gradient $p$ from the origin.}

\medskip

From $\hat{t}\geq\hat{s}$ and \eqref{line:subsolutiontest2}, we have
\begin{align*}
K+\frac{2(\overline{t}_{\alpha}-\hat{t})}{\varepsilon^{\beta}}\leq(-F)^*\left(\varepsilon X,p,\frac{\overline{x}_{\alpha}}{\varepsilon}\right)&\leq(-F)^*\left(\frac{1}{\varepsilon}Y,p,\frac{\overline{x}_{\alpha}}{\varepsilon}\right)\\
&\leq\frac{1}{\varepsilon}(-F)^*\left(Y,q,\frac{\overline{x}_{\alpha}}{\varepsilon}\right)\leq\frac{C|q|}{\varepsilon\mu}\leq\frac{C|p|}{\varepsilon^{\beta}}\\
\end{align*}
where we used \cite[Lemma 13.1]{CM14} for $(b_2,q,Y)\in\overline{\mathcal{P}}^{2,-}\overline{\ell}^{\varepsilon,\mu}(\overline{\xi}_{\alpha})$ in the second-last inequality. By \eqref{line:tildeandhat2}, we see that there exists a constant $C>0$ depending only on the data such that
\begin{align}
|p|\geq\left(C^{-1}K_1-C\right)\varepsilon^{2\beta}.\label{line:separationfromtheoriginforced2}
\end{align}
whenever $K_1>C^2$. Here, we require $\beta\leq\frac{1}{2}$.

\medskip

\textbf{Step 2.4: Deriving a contradiction for a large constant $K_1>0$.}

\medskip

Note that $\left|\frac{\overline{x}_{\alpha}}{\varepsilon}-\overline{\xi}_{\alpha}^1\right|\leq\left|\frac{\overline{x}_{\alpha}}{\varepsilon}-\overline{\xi}_{\alpha}\right|+\left|\overline{\xi}_{\alpha}-\overline{\xi}_{\alpha}^1\right|\leq\eta$ by \eqref{line:etabound2} and the fact that $\overline{\xi}_{\alpha}^1\in\overline{B}_{\mu}(\overline{\xi}_{\alpha})$. Therefore, by connecting the viscosity inequalities \eqref{line:subsolutiontest2}, \eqref{line:supersolutiontest2}, we obtain, by \eqref{line:correctors}, that
\begin{align}
C(\varepsilon^{\theta}+\varepsilon^{\beta})-\overline{F}\left(\frac{\overline{z}_{\alpha}-\hat{y}}{\varepsilon^{\beta}}\right)&\geq\lambda v^{\lambda,\eta}\left(\overline{\xi}_{\alpha}^1,\frac{\overline{z}_{\alpha}-\hat{y}}{\varepsilon^{\beta}}\right)\notag\\
&\geq-F^{\eta}\left(\frac{1}{\varepsilon}Y-2\varepsilon^{1-\beta}I_n,p-\frac{2(\varepsilon\overline{\xi}_{\alpha}^1-\overline{z}_{\alpha})}{\varepsilon^{\beta}},\overline{\xi}_{\alpha}^1\right)\notag\\
&\geq-F\left(\varepsilon X-2\varepsilon^{1-\beta}I_n,p-\frac{2(\varepsilon\overline{\xi}_{\alpha}^1-\overline{z}_{\alpha})}{\varepsilon^{\beta}},\frac{\overline{x}_{\alpha}}{\varepsilon}\right)\notag\\
&\geq-F\left(\varepsilon X,p,\frac{\overline{x}_{\alpha}}{\varepsilon}\right)+E_1+E_2\notag\\
&\geq K+\frac{\overline{t}_{\alpha}^1-\hat{s}}{\varepsilon^{\beta}}+\frac{\overline{t}_{\alpha}^1-\hat{t}}{\varepsilon^{\beta}}+E_1+E_2\label{line:chainofinequalities2}
\end{align}
where
\begin{align*}
E_1&:=F\left(\varepsilon X-2\varepsilon^{1-\beta}I_n,p,\frac{\overline{x}_{\alpha}}{\varepsilon}\right)-F\left(\varepsilon X-2\varepsilon^{1-\beta}I_n,p-\frac{2(\varepsilon\overline{\xi}_{\alpha}^1-\overline{z}_{\alpha})}{\varepsilon},\frac{\overline{x}_{\alpha}}{\varepsilon}\right),\\
E_2&:=F\left(\varepsilon X-2\varepsilon^{1-\beta}I_n,p,\frac{\overline{x}_{\alpha}}{\varepsilon}\right)-F\left(\varepsilon X,p,\frac{\overline{x}_{\alpha}}{\varepsilon}\right).
\end{align*}
Note that, from $|\varepsilon\overline{\xi}_{\alpha}^1-\overline{z}_{\alpha}|\leq|\varepsilon\overline{\xi}_{\alpha}^1-\varepsilon\overline{\xi}_{\alpha}|+|\varepsilon\overline{\xi}_{\alpha}-\overline{z}_{\alpha}|\leq C\varepsilon^{\min\left\{\frac{1}{2}+\beta,1-\theta\right\}}$ by \eqref{line:zalphazhat2}. Therefore, we have, by \eqref{line:separationfromtheoriginforced2},
\begin{align*}
\left|p+\frac{2(\varepsilon\overline{\xi}_{\alpha}^1-\overline{z}_{\alpha})}{\varepsilon^{\beta}}\cdot\nu\right|\geq\left(C^{-1}K_1-C\right)\varepsilon^{2\beta}
\end{align*}
for $\nu\in[0,1]$ if we require $K_1>C^2$, $2\beta\leq\min\left\{\frac{1}{2},1-\theta-\beta\right\}$ with a larger constant $C>0$. This implies, with \eqref{line:hessiansforced}, that
\begin{align*}
|E_1|\leq C\varepsilon^{-\beta}\left((C^{-1}K_1-C)\varepsilon^{2\beta}\right)^{-1}\varepsilon^{\min\left\{\frac{1}{2},1-\theta-\beta\right\}},
\end{align*}
and therefore, we see that there exists a constant $C>0$ depending only on the data such that if $K_1>C$, then
\begin{equation}
\begin{cases}
|E_1|\leq\frac{C}{K_1-C}\varepsilon^{\min\left\{\frac{1}{2},1-\theta-\beta\right\}-3\beta}, \quad & \\
|E_2|\leq 2n\varepsilon^{1-\beta}. \quad & 
\end{cases}\label{line:commuatatorestimateforced2}
\end{equation}

Therefore, by \eqref{line:supersolutiontest}, \eqref{line:chainofinequalities}, \eqref{line:commuatatorestimateforced}, we have
\begin{align*}
2K_1\varepsilon^{\beta}\leq C\left(\varepsilon^{\min\left\{\theta,\beta\right\}}+\frac{1}{K_1-C}\varepsilon^{\min\left\{\frac{1}{2},1-\theta-\beta\right\}-3\beta}\right)
\end{align*}
for some constant $C>0$ (with a larger one if necessary) depending only on the data. Now, we take $\beta=\frac{1}{8}$ and any $\theta\in\left[\frac{1}{8},\frac{3}{8}\right]$ as an optimal choice. Then, taking $K_1=C+1$ yields a contradiction. Therefore, there exists a constant $K_1>0$ depending only on the data such that if $\hat{t}\geq\hat{s}$, then $\hat{s}=0$.

\medskip

To prove the lower bound
$$
u^{\varepsilon}(x,t)-u(x,t)\geq -C(1+T)\varepsilon^{1/8}
$$
for all $(x,t)\in\mathbb{R}^n\times[0,T]$, we alternatively consider another auxiliary function, for a given $\varepsilon\in(0,1)$,
\begin{align*}
\Phi_1(x,y,z,t,s):=u^{\varepsilon}(x,t)&-u(y,s)-\varepsilon v^{\lambda}_{\eta}\left(\frac{x}{\varepsilon},\frac{z-y}{\varepsilon^{\beta}}\right)\\
\qquad\qquad\qquad&+\frac{|x-y|^2+|t-s|^2}{2\varepsilon^{\beta}}+\frac{|x-z|^2}{2\varepsilon^{\beta}}+K(t+s)+\gamma_1\langle y\rangle,
\end{align*}
where $\lambda=\varepsilon^{\theta},\ \eta=\varepsilon^{\beta},\ K=K_1\varepsilon^{\beta},\ \gamma_1\in(0,\varepsilon^{\beta}]$, and $\theta,\beta\in(0,1], K_1>0$ are constants to be determined. Then, the global minimum of $\Phi_1$ on $\mathbb{R}^{3n}\times[0,T]^2$ is attained at a certain point $(\hat{x},\hat{y},\hat{z},\hat{t},\hat{s})\in\mathbb{R}^{3n}\times[0,T]^2$, and we proceed estimates similarly as before.
\end{proof}

\subsection{An example}\label{subsec:example}

We prove Propositions \ref{prop:example}, \ref{prop:example2} in this subsection.

\begin{proof}[Proof of Proposition \ref{prop:example}]
As the forcing term and the initial data are radially symmetric, we have $u^{\varepsilon}(x,t)=\varphi^{\varepsilon}(r,t)$ for $r=|x|\geq0,t\geq0$, where $\varphi^{\varepsilon}$ solves
\begin{align*}
\begin{cases}
\varphi^{\varepsilon}_t-\frac{\varepsilon}{r}\varphi^{\varepsilon}_r-|\varphi^{\varepsilon}_r|=0, \quad &\text{in}\quad(0,\infty)\times(0,\infty),\\
\varphi^{\varepsilon}(r,0)=-r, \quad &\text{on}\quad[0,\infty).
\end{cases}
\end{align*}
By the optimal control formula for solutions to first-order concave Hamilton-Jacobi equations, we have
\begin{align*}
\varphi^{\varepsilon}(r,t)&=\sup\left\{-|\eta(t)|:\eta(0)=r,\ \left|\dot{\eta}(s)-\frac{\varepsilon}{\eta(s)}\right|\leq1,\ s\in[0,t]\right\}\\
&=\sup\left\{-\left|\varepsilon\xi\left(\frac{t}{\varepsilon}\right)\right|:\xi(0)=\frac{r}{\varepsilon},\ \left|\dot{\xi}(s_1)-\frac{1}{\xi(s_1)}\right|\leq1,\ s_1\in\left[0,\frac{t}{\varepsilon}\right]\right\}
\end{align*}
for $r,t\geq0$. Here, we made the changes of variables $\eta(s)=\varepsilon\xi\left(\frac{s}{\varepsilon}\right).$ and $s_1=\frac{s}{\varepsilon}$ for $s\in[0,t]$. We moreover have, for $r>t$,
\begin{align*}
\varphi^{\varepsilon}(r,t)=-\varepsilon\xi_1\left(\frac{t}{\varepsilon}\right),
\end{align*}
where $\xi_1:\left[0,\frac{t}{\varepsilon}\right]\to(0,\infty)$ is the solution to
\begin{align*}
\begin{cases}
\dot{\xi_1}(s_1)=-1+\frac{1}{\xi_1(s_1)}, \quad &\text{in}\quad(0,\infty)\times(0,\infty),\\
\xi_1(0)=\frac{r}{\varepsilon}, \quad &\text{on}\quad[0,\infty).
\end{cases}
\end{align*}

\medskip

Then, $\xi_1$ can be expressed as
\begin{align*}
\xi_1(s_1)=W\left(\left(\frac{r}{\varepsilon}-1\right)\exp\left(\frac{r}{\varepsilon}-s_1-1\right)\right)+1,
\end{align*}
where $W=W(z):[0,\infty)\to[0,\infty)$ is the Lambert W function defined by
\begin{align*}
w=W(z)\quad\Longleftrightarrow\quad z=w e^w
\end{align*}
for $z\geq0$. We can check easily that $W'(z)=\frac{W(z)}{z(1+W(z))}$ and
\begin{align*}
W(z)\geq \frac{1}{2}\log z\quad\text{for }z>0.
\end{align*}
Therefore, we immediately obtain, for $r=t>\varepsilon(1+e^{-1})$, that
\begin{align*}
-\varepsilon\xi_1\left(\frac{t}{\varepsilon}\right)\leq-\varepsilon\left(\frac{1}{2}\log\left(e^{-1}\left(\frac{t}{\varepsilon}-1\right)\right)+1\right)=-\frac{1}{2}\varepsilon\left(\log\left(\frac{t}{\varepsilon}-1\right)+1\right)<0.
\end{align*}
As $u(x,t)=0$ whenever $|x|=t$, we complete the proof.
\end{proof}

Now, we prove Proposition \ref{prop:example2}.

\begin{proof}[Proof of Proposition \ref{prop:example2}]
By \cite[Theorem 1.2]{MRR-M13}, there exists a smooth convex function $\phi$ in the variable $x'=(x_1,\cdots,x_{n-1})\in\mathbb{R}^{n-1}$ such that $\phi(x')+(\csc\alpha) t$ is a traveling wave solution to \eqref{eq:uepsilon} with $\varepsilon=1$ and with the initial datum $\phi$, and $\phi$ satisfies
\begin{align}\label{line:prop2line1}
u_0-C\leq\phi\leq u_0+C
\end{align}
with $C=2|A|\csc\alpha$, which we freeze in this proof.

Let $u^1$ be the solution to the unit scale problem \eqref{eq:uepsilon} with $\varepsilon=1$ and with the initial datum $u_0$. Applying the comparison principle to \eqref{line:prop2line1}, we obtain
\begin{align}\label{line:prop2line2}
u^1-C\leq\phi+(\csc\alpha) t\leq u^1+C.
\end{align}
Note that $u(x,t)=u_0(x)+(\csc\alpha) t$. Together with this fact, we apply \eqref{line:prop2line1} once more to \eqref{line:prop2line2} to obtain
\begin{align}\label{line:prop2line3}
|u^1(x,t)-u(x,t)|\leq 2C\quad\text{for all }(x,t)\in\mathbb{R}^n\times[0,\infty).
\end{align}
As $u_0$ is positively 1-homogeneous, we have $u^{\varepsilon}(x,t)=\varepsilon u^1\left(\frac{x}{\varepsilon},\frac{t}{\varepsilon}\right)$ and $\varepsilon u\left(\frac{x}{\varepsilon},\frac{t}{\varepsilon}\right)=u(x,t)$, which we apply to \eqref{line:prop2line3} to derive the conclusion.
\end{proof}

\section*{Acknowledgments}
The work of the author is supported in part by NSF CAREER grant DMS-1843320 and Korea Foundation for Advanced Studies.

The author would like to express his sincere gratitude to Prof. Yifeng Yu for suggesting the subject of homogenization of the curvature $G$-equation. The author also would like to express gratitude to the reviewers for the comments and the suggestions.

\section*{Conflict of interest}
The author states that there is no conflict of interest.

\section*{Data availability}
No datasets were generated or analysed during the current study.

\bibliographystyle{plain}
\bibliography{Preprint}

\begin{thebibliography}{10}

\bibitem{AC18}
Scott Armstrong and Pierre Cardaliaguet.
\newblock Stochastic homogenization of quasilinear hamilton–jacobi equations and geometric motions.
\newblock {\em J. Eur. Math. Soc.}, 20:797--864, 2018.

\bibitem{BBBL03}
Guy Barles, Samuel Biton, Mariane Bourgoing, and Olivier Ley.
\newblock Uniqueness results for quasilinear parabolic equations through viscosity solutions methods.
\newblock {\em Calc. Var. Partial Differential Equations}, 18:159--179, 2003.

\bibitem{CM14}
Luis~A. Caffarelli and Regis Monneau.
\newblock Counter-example in three dimension and homogenization of geometric motions in two dimension.
\newblock {\em Arch. Rational Mech. Anal.}, 212:503--574, 2014.

\bibitem{C-DI01}
Italo Capuzzo-Dolcetta and Hitoshi Ishii.
\newblock On the rate of convergence in homogenization of {H}amilton--{J}acobi equations.
\newblock {\em Indiana Univ. Math. J.}, 22(3):1113--1129, 09 2001.

\bibitem{CLS09}
Pierre Cardaliaguet, Pierre-Louis Lions, and Panagiotis~E. Souganidis.
\newblock A discussion about the homogenization of moving interfaces.
\newblock {\em J. Math. Pures Appl.}, 91(4):339--363, 2009.

\bibitem{CN13}
Annalisa Cesaroni and Matteo Novaga.
\newblock Long-time behavior of the mean curvature flow with periodic forcing.
\newblock {\em Comm. Partial Diff. Equa.}, 38:780--801, 2013.

\bibitem{CIL92}
Michael~G. Crandall, Hitoshi Ishii, and Pierre-Louis Lions.
\newblock User’s guide to viscosity solutions of second order partial differential equations.
\newblock {\em Bull. Am. Math. Soc.}, 27(1):1--67, 1992.

\bibitem{DKY08}
Nicolas Dirr, Georgia Karali, and Nung~K. Yip.
\newblock Pulsating wave for mean curvature flow in inhomogeneous medium.
\newblock {\em European J. Appl. Math.}, 19:661--699, 2008.

\bibitem{E89}
Lawrence~Craig Evans.
\newblock The perturbed test function method for viscosity solutions of nonlinear pde.
\newblock {\em Proc. Roy. Soc. Edinburgh Sect. A}, 111:359--375, 1989.

\bibitem{E92}
Lawrence~Craig Evans.
\newblock Periodic homogenisation of certain fully nonlinear partial differential equations.
\newblock {\em Proc. Roy. Soc. Edinburgh Sect. A}, 120(3-4):245--265, 1992.

\bibitem{GK20}
Hongwei Gao and Inwon Kim.
\newblock Head and tail speeds of mean curvature flow with forcing.
\newblock {\em Arch. Rational Mech. Anal.}, 235:287--354, 2020.

\bibitem{GLXY22}
Hongwei Gao, Ziang Long, Jack Xin, and Yifeng Yu.
\newblock Existence of effective burning velocity in cellular flow for curvature $g$-equation via game analysis.
\newblock {\em Journal of Geom. Anal.}, 2023.

\bibitem{G06}
Yoshikazu Giga.
\newblock {\em Surface Evolution Equations. A Level Set Approach. Monographs in Mathematics}, volume~99.
\newblock Birkhäuser, Basel, 2006.

\bibitem{J23}
Jiwoong Jang.
\newblock A convergence rate of periodic homogenization for forced mean curvature flow of graphs in the laminar setting.
\newblock {\em Nonlinear Differential Equations and Applications NoDEA}, 31(36), 2024.

\bibitem{LS05}
Pierre-Louis Lions and Panagiotis~E. Souganidis.
\newblock Homogenization of degenerate second-order pde in periodic and almost periodic environments and applications.
\newblock {\em Ann. Inst. H. Poincar\'{e} C, Anal. Non Lin\'{e}aire}, 22(5):667--677, 2005.

\bibitem{M51}
George~H. Markstein.
\newblock Experimental and theoretical studies of flame-front stability.
\newblock {\em J. Aero. Sci.}, 18:199--209, 1951.

\bibitem{MMTXY23}
Hiroyoshi Mitake, Connor Mooney, Hung~V. Tran, Jack Xin, and Yifeng Yu.
\newblock Bifurcation of homogenization and nonhomogenization of the curvature g-equation with shear flows.
\newblock {\em to appear in Mathematische Annalen}, 2024.

\bibitem{MTY19}
Hiroyoshi Mitake, Hung~V. Tran, and Yifeng Yu.
\newblock Rate of convergence in periodic homogenization of hamilton-jacobi equations: the convex setting.
\newblock {\em Arch. Ration. Mech. Anal.}, 233(2):901--933, 2019.

\bibitem{MRR-M13}
Regis Monneau, Jean-Michel Roquejoffre, and Violaine Roussier-Michon.
\newblock Travelling graphs for the forced mean curvature motion in an arbitrary space dimension.
\newblock {\em Ann. sci. de l'\'{E}cole Normale Sup\'{e}rieure}, 46:217--248, 2013.

\bibitem{NT05}
Hirokazu Ninomiya and Masaharu Taniguchi.
\newblock Existence and global stability of traveling curved fronts in the allen–cahn equations.
\newblock {\em J. Diff. Equa.}, 213:204--233, 2005.

\bibitem{P00}
Norbert Peters.
\newblock {\em Turbulent Combustion}.
\newblock Cambridge University Press, 2000.

\bibitem{QSTY23}
Jianliang Qian, Timo Sprekeler, Hung~V. Tran, and Yifeng Yu.
\newblock Optimal rate of convergence in periodic homogenization of viscous hamilton-jacobi equations.
\newblock {\em preprint}, 2024.

\bibitem{T21}
Hung~V. Tran.
\newblock {\em Hamilton--Jacobi equations: theory and applications}, volume 213.
\newblock AMS Graduate Studies in Mathematics, 2021.

\bibitem{TY23}
Hung~V. Tran and Yifeng Yu.
\newblock Optimal convergence rate for periodic homogenization of convex hamilton-jacobi equations.
\newblock {\em Indiana Univ. Math J.}, 2023.

\bibitem{R94}
J~Zhu and P.~D. Ronney.
\newblock Simulation of front propagation at large non-dimensional flow disturbance intensities.
\newblock {\em Combust. Sci. Technol.}, 100:183--201, 1994.

\end{thebibliography}

\end{document}